\documentclass[12pt]{article}

\pdfoutput=1

\usepackage[a4paper]{geometry}
\usepackage{amsthm}
\usepackage{amsmath}
\usepackage{amssymb}
\usepackage{amsfonts}
\usepackage{graphicx}
\usepackage[colorlinks=true,citecolor=black,linkcolor=black,urlcolor=blue]{hyperref}
\newcommand{\doi}[1]{\href{http://dx.doi.org/#1}{\texttt{doi:#1}}}
\newcommand{\arxiv}[1]{\href{http://arxiv.org/abs/#1}{\texttt{arXiv:#1}}}

\theoremstyle{plain}
\newtheorem{theorem}{Theorem}
\newtheorem{lemma}[theorem]{Lemma}
\newtheorem{corollary}[theorem]{Corollary}
\newtheorem{proposition}[theorem]{Proposition}

\theoremstyle{definition}
\newtheorem{definition}[theorem]{Definition}
\newtheorem{example}[theorem]{Example}

\theoremstyle{remark}

\renewcommand{\emptyset}{\varnothing}

\title{Isotropic matroids I\@: Multimatroids and neighborhoods}

\author{Robert Brijder\thanks{R.B.\ is a postdoctoral fellow of the Research Foundation -- Flanders (FWO).}\\
\small Hasselt University\\[-0.8ex]
\small Belgium\\
\small\tt robert.brijder@uhasselt.be\\
\and
Lorenzo Traldi\\
\small Lafayette College\\[-0.8ex]
\small Easton, Pennsylvania, U.S.A.\\
\small\tt traldil@lafayette.edu
}

\date{}

\begin{document}

\maketitle

\begin{abstract}
Several properties of the isotropic matroid of a looped simple graph are
presented. Results include a characterization of the multimatroids that are
associated with isotropic matroids and several ways in which the isotropic
matroid of $G$ incorporates information about graphs locally equivalent to
$G$. Specific results of the latter type include a characterization of graphs
that are locally equivalent to bipartite graphs, a direct proof that two
forests are isomorphic if and only if their isotropic matroids are isomorphic,
and a way to express local equivalence indirectly, using only edge pivots.

\bigskip\noindent \textbf{Keywords:} delta-matroid, interlacement, isotropic system, local equivalence,
matroid, multimatroid, stable set

\end{abstract}

\section{Introduction}

Let $G$ be a looped simple graph, i.e., a graph in which no two edges are incident on precisely the same set of vertices. Although we allow loops, we reserve the terms \emph{adjacent} and \emph{neighbors} for pairs of distinct vertices. We do not count loops in vertex degrees, and we do not consider $v$ to be an element of the open neighborhood $N_{G}(v)=\{w \in V(G) \mid vw \in E(G)\}$, whether $v$ is looped or not. Also, in this paper the rows and columns of matrices are not ordered, but are instead indexed by some finite sets $X$ and $Y$, respectively; we refer to such a matrix as an $X \times Y$ matrix. A conventional matrix is then just a $\{1,\ldots,m\} \times \{1,\ldots,n\}$ matrix. We remark that we follow this more general convention because we will consider adjacency matrices of graphs which do not have a canonical linear ordering of their columns and rows (because graphs do not have a canonical linear ordering of their vertices). Using instead conventional matrices and fixing an arbitrary linear ordering is notationally more cumbersome due to the frequent need for permutation matrices to permute the rows and columns. In situations where there is an obvious natural bijection between the rows and columns of a matrix $A$, we may, e.g., refer to the ``diagonal'' of $A$.


Recently, the second author introduced the binary matroid represented over $GF(2)$ (the $2$-element field) by the matrix%
\[
IAS(G)=%
\begin{pmatrix}
I & A(G) & A(G)+I
\end{pmatrix}
\text{,}%
\]
where $I$ is the identity matrix and $A(G)$ is the adjacency matrix of $G$, i.e., the $V(G) \times V(G)$ binary matrix with diagonal entries equal to $1$ for looped vertices, and off-diagonal entries equal to $1$ for adjacent vertices~\cite{Tnewnew}. (For each $v \in V(G)$, the $v$ rows of $I$, $A(G)$ and $A(G)+I$ constitute the $v$ row of $IAS(G)$.) This matroid is called the \emph{isotropic matroid} of $G$, and denoted $M[IAS(G)]$. We refer to Oxley's book~\cite{O} for terminology regarding matroids; we do not repeat the definitions of basic notions like circuits, connectedness, independence, rank etc.\ except when these notions require special attention for isotropic matroids.

The purpose of the present paper is to present extensions of the discussion of isotropic matroids in \cite{Tnewnew}. Before discussing these extensions, we establish some notation and terminology. The columns of $IAS(G)$ are labeled as follows: the $v$ column of $I$ is designated $\phi_{G}(v)$, the $v$ column of $A(G)$ is designated $\chi_{G}(v)$, and the $v$ column of $I+A(G)$ is designated $\psi_{G}(v)$. The set $\{\phi_{G}(v)$, $\chi_{G}(v)$, $\psi_{G}(v)\mid v\in V(G)\}$ is denoted $W(G)$; it is the ground set of the isotropic matroid $M[IAS(G)]$. The matroid $M[IAS(G)]$ reflects algebraic interactions among the columns of $IAS(G)$: A subset of $W(G)$ is dependent (or independent, or a basis) if and only if the corresponding set of columns of $IAS(G)$ is linearly dependent (or independent, or a basis of the column space of $IAS(G)$); $M[IAS(G)]$ is the direct sum of two submatroids if and only if the two corresponding sets of columns partition $W(G)$ and their linear spans share only the zero vector; and so on. If $v\in V(G)$ then the subset $\tau_{G}(v)=\{\phi_{G}(v)$, $\chi_{G}(v)$, $\psi_{G}(v)\}$ of $W(G)$ is the \emph{vertex triple} corresponding to $v$. Note that the three columns of $IAS(G)$ corresponding to $\tau_{G}(v)$ sum to $0$, so every vertex triple is a dependent set of $M[IAS(G)]$. If $v$ is not isolated then each of the three corresponding columns of $IAS(G)$ has a nonzero entry, so $\tau_{G}(v)$ is a circuit of $M[IAS(G)]$. If $v$ is isolated, instead, then one of $\chi_{G}(v),\psi_{G}(v)$ is a loop of $M[IAS(G)]$, and the other is parallel to $\phi_{G}(v)$ in $M[IAS(G)]$.

A subset $S\subseteq W(G)$ is a \emph{subtransversal} if it contains no more than one element of each vertex triple; if $S$ contains precisely one element from each vertex triple, it is a \emph{transversal}. The families of subtransversals and transversals of $G$ are denoted $\mathcal{S}(G)$ and $\mathcal{T}(G)$, respectively. A \emph{transverse matroid} of $G$ is a submatroid obtained by restricting $M[IAS(G)]$ to a transversal; we use ``transverse matroid'' to avoid confusion with transversal matroids. A \emph{transverse circuit} of $G$ is a circuit of a transverse matroid, i.e., a subtransversal that is a circuit of $M[IAS(G)]$.

In Section~\ref{sec:char_isotr_mm} we provide natural abstract properties that characterize isotropic matroids and related structures. Indeed, given a binary matroid $M$ and a partition $\Omega$ of the ground set of $M$ into sets of cardinality $3$, one may wonder what the essential properties of $M$ are such that $(M,\Omega)$ is of the form $(M[IAS(G)],W(G))$ for some graph $G$. To elegantly expose these essential properties, we use the theory of multimatroids.

We explain the connection between isotropic matroids and multimatroids using the notion of a sheltering matroid, which was mentioned in passing by Bouchet~\cite{B3}. A sheltering matroid is a matroid such that its set of transverse matroids forms a multimatroid. In particular, an isotropic matroid is a 3-sheltering matroid with respect to the partition of the ground set into vertex triples. We define representability of sheltering matroids over a field $\mathbb{F}$ and we characterize the isotropic matroids among the $GF(2)$-representable 3-sheltering matroids (cf. Theorem~\ref{thm:3shelt_isotrm}). As an isotropic matroid is uniquely determined by its multimatroid, we also characterize the multimatroids corresponding to isotropic matroids in terms of natural properties of multimatroids (cf. Theorem~\ref{thm:char_bt3sm}). We moreover consider a stronger notion of representability for 2-matroids and 2-sheltering matroids inspired by the usual notion of representability for delta-matroids.

After discussing sheltering matroids in Section~\ref{sec:char_isotr_mm}, we devote the rest of the paper to the relationship between the matroidal structure of $M[IAS(G)]$ and the graphical structures of $G$ and locally equivalent graphs. We need some more notation and terminology to describe these results.

\begin{definition}
\label{localeq}
\begin{enumerate}
\item If $v\in V(G)$ then the graph obtained from $G$ by complementing the loop status of $v$ is denoted $G_{\ell}^{v}$.

\item If $v\in V(G)$ then the graph obtained from $G$ by complementing the adjacency status of every pair of neighbors of $v$ is denoted $G_{s}^{v}$, and it is called the \emph{simple local complement} of $G$ with respect to $v$.

\item If $v\in V(G)$ then the graph obtained from $G$ by complementing the adjacency status of every pair of neighbors of $v$ and the loop status of every neighbor of $v$ is denoted $G_{ns}^{v}$, and it is called the \emph{non-simple local complement} of $G$ with respect to $v$.

\item A graph that can be obtained from $G$ using loop complementations and local complementations is \emph{locally equivalent} to $G$. \label{item:def_loc_equiv}
\end{enumerate}
\end{definition}

We should mention that $G_{s}^{v}$ and $G_{ns}^{v}$ are both called ``local complements'' of $G$ in the literature. For precision we use the unmodified term ``local complement'' only in situations like item~\ref{item:def_loc_equiv} of Definition~\ref{localeq}, where both types of local complement are included.

Also, we note that according to Definition~\ref{localeq}, locally equivalent graphs have the same vertices. Consequently the equivalence relation on graphs generated by isomorphism and local equivalence is strictly coarser than isomorphism or local equivalence alone.

The principal result of~\cite{Tnewnew} is that two graphs are locally equivalent up to isomorphism if and only if their isotropic matroids are isomorphic. In fact, a sequence of local complementations and loop complementations that transforms $G$ into an isomorphic copy of $H$ will directly induce a corresponding isomorphism $M[IAS(G)]\to M[IAS(H)]$; see Section~\ref{sec:isom_isotropic_mat} for details.

\begin{definition}
\label{neighcirc}Let $v$ be a vertex of $G$, with open neighborhood $N_{G}(v)$.
Then the \emph{neighborhood circuit }of $v$, denoted $\zeta_{G}(v)$, is
$\{\chi_{G}(v)\}\cup\{\phi_{G}(w)\mid w\in N_{G}(v)\}$ if $v$ is unlooped, or
$\{\psi_{G}(v)\}\cup\{\phi_{G}(w)\mid w\in N_{G}(v)\}$ if $v$ is looped.
\end{definition}

Notice that whichever of $\chi_{G}(v),\psi_{G}(v)$ is included in $\zeta_{G}(v)$, the nonzero entries of the corresponding column of $IAS(G)$ appear in the rows corresponding to neighbors of $v$; hence $\zeta_{G}(v)$ is a transverse circuit of $M[IAS(G)]$. Also, if $\Phi(G)=\{\phi_{G}(v)\mid v\in V(G)\}$ then $G$ is determined up to isomorphism by the submatroid of $M[IAS(G)]$
whose ground set is
\[
\zeta(G)=\Phi(G) \cup \bigcup\limits_{v\in V(G)} \zeta_{G}(v)\text{.}
\]
For if $v\in V(G)$ then $v$ is looped if and only if $\psi_{G}(v)\in\zeta(G)$, and the open neighborhood $N_{G}(v)$ is determined by the fundamental circuit of $\chi_{G}(v)$ or $\psi_{G}(v)$ (whichever is included in $\zeta(G)$) with respect to the basis $\Phi(G)$.

Recall that a subset $X\subseteq V(G)$ is \emph{stable} if no two elements of $X$ are neighbors in $G$. (We consider all sets of cardinality $0$ or $1$ to be stable.) By the way, stable sets are also called ``independent'' but we do not use that term here, to avoid any possibility of confusion with matroid independence. 

\begin{definition}
\label{neighmat}If $X$ is a stable set of $G$ then
\[
T_{G}(X)=\{\phi_{G}(v)\mid v\notin X\} \cup \{\chi_{G}(x)\mid x\in X\text{ is unlooped}\}\cup \{\psi_{G}(x)\mid x\in X\text{ is looped}\}
\]
is a transversal of $W(G)$. We call the restriction $M[IAS(G)]\mid T_{G}(X)$ the \emph{neighborhood matroid} of $X$ in $G$, and denote it $M_{G}(X)$.
\end{definition}

Notice that $T_{G}(X)$ contains the neighborhood circuits of the elements of $X$.

A looped simple graph $G$ may certainly have transverse circuits that are not neighborhood circuits, and transverse matroids that are not neighborhood matroids. However, it turns out that all transverse circuits and transverse matroids correspond to neighborhood circuits and neighborhood matroids in graphs locally equivalent to $G$:

\begin{theorem}
\label{tranmat}Let $T\in\mathcal{T}(G)$. Then there is a graph $H$ locally
equivalent to $G$, with the property that an induced isomorphism
$M[IAS(G)]\rightarrow M[IAS(H)]$ maps the transverse matroid $M[IAS(G)]\mid T$
isomorphically to a neighborhood matroid of a stable set of $H$.
\end{theorem}

\begin{theorem}
\label{trancirc}Let $S\in\mathcal{S}(G)$. Then $S$ is a transverse circuit of
$G$ if and only if there is a graph $H$ locally equivalent to $G$, with the
property that an induced isomorphism $M[IAS(G)]\rightarrow M[IAS(H)]$ maps $S$
to a neighborhood circuit of $H$.
\end{theorem}

Here are two direct consequences of Theorems \ref{tranmat} and
\ref{trancirc}.

\begin{corollary}
\label{nullnu}
Let $G$ be a looped simple graph, and $\nu$ a positive integer. Then $G$ has a transverse matroid of nullity $\nu$ if and only if some graph locally
equivalent to $G$ has a stable set of size $\nu$.
\end{corollary}

\begin{corollary}
\label{circk}Suppose $G$ is a looped simple graph, and $k\in\mathbb{N}$. Then $G$ has a transverse circuit of size $k$ if and only if some graph locally equivalent to $G$ has a vertex of degree $k-1$.
\end{corollary}

These results indicate the close relationship between $M[IAS(G)]$ and the structures of graphs locally equivalent to $G$. A special case of Theorem \ref{tranmat} also provides a simple explanation of the fact that $M[IAS(G)]$ determines $G$ up to isomorphism and local equivalence~\cite{Tnewnew}: in fact, all of the graphs included in the local equivalence class of $G$ are determined up to isomorphism by $M[IAS(G)]$. As detailed in Corollary \ref{allphibasis}, if $H$ is locally equivalent to $G$ and a local equivalence induces a matroid isomorphism $\beta:M[IAS(G)]\rightarrow M[IAS(H)]$, then vertex neighborhoods in $H$ correspond directly to fundamental circuits in $M[IAS(G)]$ with respect to the basis $\beta^{-1}(\Phi(H))$.

In Sections~\ref{sec:bipartite} and \ref{wheel} we discuss two more ways to use the results above: one is a characterization of graphs that are locally equivalent to bipartite graphs, and the other is a characterization of the local equivalence class of the wheel graph $W_5$.

In Section~\ref{sec:minors} we discuss minors of isotropic matroids. Section~\ref{sec:par_red_hereditary} is focused on a special type of minor: a parallel reduction. It turns out that parallel reductions of isotropic matroids correspond precisely to pendant-twin reductions of graphs. In particular, the graphs whose isotropic matroids can be resolved using parallel reductions are the same as the graphs that can be resolved using pendant-twin reductions. These are the graphs whose connected components are all distance hereditary \cite{BM}.

As a special case, in Section~\ref{sec:forests} we prove the following striking result, which underscores the fundamental difference between isotropic matroids of graphs and the more familiar cycle matroids.

\begin{theorem}
\label{forests}Two forests are isomorphic if and only if their isotropic matroids are isomorphic.
\end{theorem}

\subsection{Remarks about delta-matroids and isotropic systems}

Before providing details of our results, we briefly describe the connections tying the two kinds of matroid structures we discuss in detail (isotropic matroids and multimatroids) to two other kinds of structures (delta-matroids and isotropic systems), which were introduced earlier. Three remarks about these structures will provide some context.


(i) Beginning in the 1980s, Bouchet~\cite{bouchet1987} and other authors developed a general theory of delta-matroids, which includes delta-matroids associated with graphs and delta-matroids not associated with graphs. The delta-matroids associated with graphs are \emph{binary}, i.e., they can be represented (in an appropriate sense) over $GF(2)$. Bouchet introduced isotropic systems at about the same time~\cite{Bi1, Bi2}. In contrast with the theory of delta-matroids, a general theory of isotropic systems that would include instances not represented over $GF(2)$ has not been fully developed, though it has been introduced by other authors~\cite{BB}. This contrast is reflected in terminology: the term ``delta-matroid'' does not include an assumption that the structure is tied to $GF(2)$, but the term ``isotropic system'' does include such an assumption. Isotropic matroids are essentially equivalent to isotropic systems~\cite{Tnewnew}, and are named for them.

(ii) In the 1990s Bouchet introduced multimatroids~\cite{B1, B2, B3, B4}, as a way of providing a common generalization of the theories of delta-matroids and isotropic systems. Delta-matroids are equivalent to multimatroids of a particular type, the 2-matroids, and isotropic systems are equivalent to multimatroids of a different particular type, a subclass of the 3-matroids. A looped simple graph has a corresponding 2-matroid and also a corresponding 3-matroid; the 2-matroid is equivalent to the graph's delta-matroid, and the 3-matroid is equivalent to the graph's isotropic matroid and isotropic system. Consequently when we explicitly discuss the 2-matroids and isotropic matroids of graphs, we are also implicitly discussing the delta-matroids, 3-matroids and isotropic systems of graphs.

(iii) More recently, Brijder and Hoogeboom have observed that some delta-matroids admit a loop complementation operation. They call these delta-matroids ``vf-safe.'' The class of vf-safe delta-matroids properly contains the class of binary delta-matroids; for instance all quaternary matroids are vf-safe~\cite{BH4}. In~\cite{BH3} loop complementation is used to show that the 2-matroid corresponding to a vf-safe delta-matroid extends to a special type of 3-matroid in a canonical way. For the binary delta-matroid associated to a graph $G$, the delta-matroid loop complementation operation is compatible with graph-theoretic loop complementation. Moreover, if the construction of~\cite{BH3} is applied to the binary delta-matroid associated with a graph $G$, the result is the 3-matroid associated with $G$. Consequently the 2-matroid, the delta-matroid, the isotropic system, the 3-matroid and the isotropic matroid of a graph are all essentially equivalent to each other.

It might seem strange to try to explain the connections tying together four types of objects --- graphs, binary delta-matroids, isotropic systems, and multimatroids --- by introducing isotropic matroids into an already complicated situation. But there are three natural reasons to expect isotropic matroids to yield useful insights. One reason is that the relationship between a graph and its isotropic matroid is fairly transparent, as $M[IAS(G)]$ is defined directly from the adjacency matrix of $G$. The second reason is that unlike delta-matroids, isotropic systems and multimatroids, which are specialized types of structures, isotropic matroids are ordinary binary matroids. The theory of binary matroids has been developed thoroughly since Whitney introduced matroids more than 80 years ago, and this theory can be applied directly to isotropic matroids. The third reason is that $M[IAS(G)]$ contains the binary delta-matroid, isotropic system and multimatroid associated with $G$, so we can see the interactions among these structures within the isotropic matroid.

In summary, we see that although the connections among delta-matroids, isotropic systems and multimatroids are quite complicated in general, the theories are very closely related when restricted to instances representable over $GF(2)$. The following compilation of results from various references indicates that this close relationship also includes isotropic matroids, and that all these structures detect local equivalence.

\begin{theorem}
\label{iso1} If $G$ and $H$ are looped simple graphs then any one of the following implies the rest.
\begin{enumerate}
\item $G$ and $H$ are locally equivalent, up to isomorphism.
\item Up to isomorphism, the binary delta-matroid associated to $H$ may be obtained from the binary delta-matroid associated to $G$ by applying some twists and loop complementations.
\item The isotropic systems associated to $G$ and $H$ are strongly isomorphic.
\item The 3-matroids associated to $G$ and $H$ are isomorphic.
\item The isotropic matroids associated to $G$ and $H$ are isomorphic.
\end{enumerate}
\end{theorem}

\section{Sheltering matroids and their representability}

\label{sec:char_isotr_mm}

In this section we define the notion of sheltering matroid and show its
relationship with the notion of multimatroid from the literature.

\subsection{Multimatroids}

\label{ssec:mmatroids}

We now recall the notion of multimatroid and related notions from \cite{B1}.
Let $\Omega$ be a partition of a finite set $U$. A $T\subseteq U$ is called a
\emph{transversal} (\emph{subtransversal}, respectively) of $\Omega$ if $|T\cap
\omega|=1$ ($|T\cap\omega|\leq1$, respectively) for all $\omega\in\Omega$. We denote
the set of transversals of $\Omega$ by $\mathcal{T}(\Omega)$ and the set of
subtransversals of $\Omega$ by $\mathcal{S}(\Omega)$. A $p\subseteq U$ is
called a \emph{skew pair} of $\omega\in\Omega$ if $|p|=2$ and $p\subseteq
\omega$. We say that $\Omega$ is a \emph{$q$-partition} if $q=|\omega|$ for
all $\omega\in\Omega$. A \emph{transversal $q$-tuple} of a $q$-partition
$\Omega$ is a sequence $\tau=(T_{1},\ldots,T_{q})$ of $q$ mutually disjoint
transversals of $Q$. Note that the elements of $\tau$ are ordered.

Multimatroids form a generalization of matroids. Like matroids, multimatroids
can be defined in terms of rank, circuits, independent sets, etc. Here they
are defined in terms of independent sets.

\begin{definition}
[\cite{B1}]\label{def:multimatroid} Let $\Omega$ be a partition of a finite
set $U$. A \emph{multimatroid} $Z$ over $(U,\Omega)$, described by its
independent sets, is a triple $(U,\Omega,\mathcal{I})$, where $\mathcal{I}%
\subseteq\mathcal{S}(\Omega)$ is such that:

\begin{enumerate}
\item for each $T \in\mathcal{T}(\Omega)$, $(T,\mathcal{I}\cap2^{T})$ is a
matroid (described by its independent sets) and

\item for any $I \in\mathcal{I}$ and any skew pair $p = \{x,y\}$ of some
$\omega\in\Omega$ with $\omega\cap I = \emptyset$, $I \cup\{x\} \in
\mathcal{I}$ or $I \cup\{y\} \in\mathcal{I}$.
\end{enumerate}
\end{definition}

A multimatroid $Z$ is said to be \emph{nondegenerate} if $|\omega|>1$ for all
$\omega\in\Omega$. If $\Omega$ is a $q$-partition, then we say that $Z$ is a
\emph{$q$-matroid}. If $Z$ is a $1$-matroid, then we also view $Z$ simply as a
matroid. A \emph{basis} of a multimatroid $Z$ is a set in $\mathcal{I}$
maximal with respect to inclusion. It is shown in \cite{B1} that the bases of
a nondegenerate multimatroid are of cardinality $|\Omega|$. We say that $C
\in\mathcal{S}(\Omega)$ is a \emph{circuit} if $C$ is not an independent set
and $C$ is minimal with this property (with respect to inclusion). For
$X\subseteq U$, we define $Z[X]=(X,\Omega^{\prime},\mathcal{I}^{\prime})$ with
$\Omega^{\prime}=\{\omega\cap X\mid\omega\in\Omega,\omega\cap X\neq
\emptyset\}$ and $\mathcal{I}^{\prime}=\{I\in\mathcal{I}\mid I\subseteq X\}$.
We also define $Z-X=Z[U-X]$. Moreover, $Z$ is called \emph{tight} if both
$Z$ is nondegenerate and for every $S\in\mathcal{S}(\Omega)$
with $|S|=|\Omega|-1$, there is an $x\in\omega$ such that the rank of the
matroid $Q[S]$ (recall that we associate a 1-matroid with a matroid) is equal
to the rank of the matroid $Q[S\cup\{x\}]$, where $\omega$ is the unique set
in $\Omega$ such that $S\cap\omega=\emptyset$.

\subsection{Sheltering matroids}

Recall the notion of sheltering matroid, which was mentioned in the introduction.

\begin{definition}
\label{shelter}A \emph{sheltering matroid} is a tuple $Q=(M,\Omega)$ where $M$
is a matroid over some ground set $U$ and $\Omega$ is a partition of $U$, such
that for any independent set $I\in\mathcal{S}(\Omega)$ of $M$ and for any skew
pair $p=\{x,y\}$ of $\omega\in\Omega$ with $\omega\cap I=\emptyset$,
$I\cup\{x\}$ or $I\cup\{y\}$ is an independent set of $M$.
\end{definition}

Many matroid notions carry over straightforwardly to sheltering matroids. For
example, for $X\subseteq U$, we define the \emph{deletion} of $X$ from $Q $ by
$Q-X=(M-X,\Omega^{\prime})$ with $\Omega^{\prime}=\{\omega\setminus
X\mid\omega\in\Omega, \omega\setminus X \neq\emptyset\}$.

Note that if $Q=(M,\Omega)$ is a sheltering matroid, then $\mathcal{Z}%
({Q})=(U,\Omega,\mathcal{I})$ with $U$ the ground set of $M$ and
$\mathcal{I}=\{I\in\mathcal{S}(\Omega)\mid
I\mbox{ is an independent set of }M\}$ is a multimatroid. We say that
$\mathcal{Z}({Q})$ is the \emph{multimatroid corresponding} to $Q$. Also, we
say that $Q$ (or $M$) \emph{shelters} the multimatroid $\mathcal{Z}({Q})$. Not
every multimatroid is sheltered by a matroid \cite{B1}. Note that for
$X\subseteq U$, $\mathcal{Z}(Q-X)=\mathcal{Z}(Q)-X$. If $\mathcal{Z}({Q})$ is
a $q$-matroid, then $Q$ is called a \emph{$q$-sheltering matroid}, and $Q$ is
called \emph{tight} if $\mathcal{Z}(Q)$ is tight. It follows from \cite[Proposition~41]%
{Tnewnew} that $M[IAS(G)]$ is a tight 3-sheltering matroid, with $\Omega$ the
partition of $W(G)$ into vertex triples.

Let $Q_{1}=(M_{1},\Omega_{1})$ and $Q_{2}=(M_{2},\Omega_{2})$ be sheltering
matroids. An \emph{isomorphism} $\varphi$ from $Q_{1}$ to $Q_{2}$ is an
isomorphism from $M_{1}$ to $M_{2}$ that respects the skew classes, i.e., if
$x$ and $y$ are elements of the ground set of $M_{1}$, then $x$ and $y$ are in
a common skew class of $\Omega_{1}$ if and only if $\varphi(x)$ and
$\varphi(y)$ are in a common skew class of $\Omega_{2}$. If $Q_{1}$ and $Q_{2}
$ are isomorphic then $\mathcal{Z}({Q}_{1})$ and $\mathcal{Z}({Q}_{2})$ are
isomorphic too; but the converse is far from true:

\begin{example}
\label{example}Let $U=\{\phi_{1},\phi_{2},\chi_{1},\chi_{2}\}$ and
$\Omega=\{\{\phi_{1},\chi_{1}\},\{\phi_{2},\chi_{2}\}\}$. Let $Z$ be the
multimatroid in which every element of $\mathcal{S}(\Omega)$ is independent.
Then $Z$ has several nonisomorphic sheltering matroids, including the uniform
matroids $U_{4,4}$, $U_{3,4}$, $U_{2,4}$ and the matroid with bases
$\{\phi_{1},\phi_{2}\}$, $\{\phi_{1},\chi_{2}\}$, $\{\chi_{1},\phi_{2}\}$ and
$\{\chi_{1},\chi_{2}\}$.
\end{example}

Note that in Example~\ref{example} there are sheltering matroids of ranks $2$,
$3$ and $4$. In general, if $Q=(M,\Omega)$ is a sheltering matroid with
$\mathcal{Z}(Q)$ nondegenerate, then $M$ is of rank $r(M)\geq\left\vert
\Omega\right\vert $, as all bases of $\mathcal{Z}(Q)$ are independent in $M$.
Moreover, if $Q$ is a sheltering matroid with $r(M)>|\Omega|$, then a
sheltering matroid $Q^{tr}=(M^{tr},\Omega)$ is obtained from $Q$ by
truncation: $M^{tr}$ is the matroid whose independent sets are the independent sets of $M$ of cardinality $<r(M)$. By truncating repeatedly, we conclude that a nondegenerate multimatroid $Z$ can be sheltered by a matroid if and only if $Z$ can be sheltered by a matroid of rank $|\Omega|$.

\begin{definition}
We say that a sheltering matroid $Q = (M,\Omega)$ is \emph{strict} if $r(M)\leq|\Omega|$.
\end{definition}

If $Q$ is nondegenerate, the condition $r(M) \leq|\Omega|$ is equivalent to saying that the family of bases of $M$ that are (sub)transversals is equal to the family of bases of $\mathcal{Z}(Q)$. In particular, $r(M) \leq|\Omega|$ is equivalent to $r(M) = |\Omega|$.

\subsection{Representable multimatroids and sheltering matroids}

We say that a sheltering matroid $Q=(M,\Omega)$ is \emph{representable} over
the field $\mathbb{F}$ if the matroid $M$ is representable over $\mathbb{F}$.
We say that a multimatroid $Z$ is \emph{representable} over $\mathbb{F}$ if
there is a sheltering matroid $Q$ representable over $\mathbb{F}$ that
shelters $Z$. Note that this notion of representability for 1-matroids
corresponds to the usual notion of representability for matroids.

One might define a weaker version of representability for sheltering matroids
$Q=(M,\Omega)$ (and multimatroids $Z$) by requiring only that $Z$ defines
$\mathbb{F}$-representable matroids on the transversals of $\Omega$; Bouchet
and Duchamp presented a similar definition in \cite{Bouchet_1991_67}. We do
not explore this weaker version of representability in this paper.

We say that a multimatroid $Z$ is \emph{strictly representable} over
$\mathbb{F}$ if there is a strict sheltering matroid $Q$ representable over
$\mathbb{F}$ that shelters $Z$.

We say that a sheltering matroid and multimatroid are \emph{binary} when they
are representable over $GF(2)$. Similarly, we say that a multimatroid is
\emph{strictly binary} if it is strictly representable over $GF(2)$. In this
subsection we consider mainly 2-sheltering matroids and 2-matroids, and in
particular binary 2-sheltering matroids and binary 2-matroids.

Let $A$ be a $V\times V$ matrix (i.e., $A$ is a $|V| \times |V|$ matrix where the rows and columns are not ordered, but instead indexed by $V$). The principal pivot transform
\cite{Tsatsomeros2000151} of $A$ with respect to $X\subseteq V$ with $A[X]$
nonsingular is a $V\times V$ matrix denoted by $A\ast X$. We do not detail the
definition of principal pivot transform here, but we recall three useful
properties. The first of these properties is that if
\[
E=\bordermatrix{
& B & T \cr
& I & A
}
\]
is a standard representation of some matroid $M$ with respect to a basis $B$,
and $B^{\prime}$ is another basis of $M$, then
\[
E^{\prime}=\bordermatrix{
& B' & T \Delta B' \Delta B \cr
& I & A*(B' \cap T)
}
\]
is a standard representation of $M$ with respect to $B^{\prime}$. To state the
second property, recall that a matrix $A$ is skew-symmetric if $A^{T}=-A$.
Thus, skew-symmetric matrices over fields of characteristic $2$ may have
nonzero diagonal entries. The second useful property of the principal pivot
transform is that if $A$ is skew-symmetric, so is $A\ast X$. The third useful
property is that if $A$ is skew-symmetric and zero-diagonal, so is $A\ast X$.

The following lemma is from \cite[Theorem~4.1]{B1}.

\begin{lemma}
[\cite{B1}]\label{lem:char_bases_2matroid} Let $\Omega$ be a $2$-partition of
$U$, and $\mathcal{B}$ a nonempty subset of $\mathcal{T}(\Omega)$. Then $\mathcal{B}$ is the
set of bases of a 2-matroid over $(U,\Omega)$ if and only if for all
$B,B^{\prime}\in\mathcal{B}$ and $p \subseteq B \Delta B^{\prime}$ a skew
pair, there is a skew pair $q \subseteq B \Delta B^{\prime}$ such that $B
\Delta(p \cup q) \in\mathcal{B}$ (we allow $p=q$).
\end{lemma}

The following lemma is essentially from \cite{bouchet1987} from the context of delta-matroids. Recall the definition of transversal $q$-tuple from Subsection~\ref{ssec:mmatroids}.

\begin{lemma}
[\cite{bouchet1987}]\label{lem:repr_implies_2smatroid} Let $\tau=(T_{1}%
,T_{2})$ be a transversal 2-tuple of $\Omega$, let
\[
E=\bordermatrix{
& T_1 & T_2 \cr
& I & A
}
\]
be a matrix with $A$ a skew-symmetric matrix over some field $\mathbb{F}$, and
let $M$ be the column matroid of $E$. Then $Q=(M,\Omega)$ is a $2$-sheltering
matroid.
\end{lemma}

\begin{proof}
To show that $\mathcal{Z}(Q)$ is a 2-matroid, we invoke
Lemma~\ref{lem:char_bases_2matroid}. Let $B_{1}$ and $B_{2}$ be bases of $M$,
which are transversals of $\Omega$, and let $p\subseteq B_{1}\Delta B_{2}$ be
a skew pair. By applying principal pivot transform, we have that $M$ is
represented by
\[
E^{\prime}=\bordermatrix{
& B_1 & T \cr
& I & A'
}
\]
for some skew-symmetric matrix $A^{\prime}$ and some $T\in\mathcal{T}(\Omega
)$. Let $p=\{p_{1},p_{2}\}$ with $p_{1}\in B_{1}$. If $B_{1}\Delta
p\notin\mathcal{B}$, then the diagonal entry $A^{\prime}[\{p_{2}\}]$ is zero.
Since $B_{2}$ is a basis, the column of $p_{2}$ in $E$ is nonzero. Thus there
is a $q_{2}\in T$ such that
\[
A^{\prime}[\{p_{2},q_{2}\}]=\bordermatrix{
& p_2 & q_2 \cr
p_2 & 0 & x \cr
q_2 & -x & y
}
\]
for some $x\in\mathbb{F}\setminus\{0\}$ and $y\in\mathbb{F}$. Since
$A^{\prime}[\{p_{2},q_{2}\}]$ is nonsingular, we have $B_{1}\Delta(p\cup
q)\in\mathcal{B}$, where $q$ is the skew pair containing $q_{2}$.
\end{proof}

We denote $Q$ of Lemma~\ref{lem:repr_implies_2smatroid} by $\mathcal{Q}%
(A,\tau,2)$.

\begin{lemma}
\label{lem:symmetric_iff_2sm} Let
\[
E=\bordermatrix{
& B & T \cr
& I & A
}
\]
be a matrix over $GF(2)$, where $A$ is zero-diagonal. Let $\Omega$ be the
natural 2-partition such that $B$ and $T$ are transversals of $\Omega$. Then
$E$ represents a 2-sheltering matroid if and only if $A$ is symmetric.
\end{lemma}

\begin{proof}
The if direction follows from Lemma~\ref{lem:repr_implies_2smatroid}. For the
only-if direction, assume to the contrary that $A$ is not symmetric. Then
there are $a,b\in T$ such that $A[\{a,b\}]$ is of the form
\[
\bordermatrix{
& a & b \cr
a & 0 & 0 \cr
b & 1 & 0
}.
\]
Consider $I=(B\setminus(\omega_{a}\cup\omega_{b}))\cup\{a\}$, where
$\omega_{x}$ is the skew class of $\Omega$ containing $x\in U$. Let $M$ be the
matroid represented by $E$. Note that $I$ is an independent set of $M$.
However, there is no $x\in\omega_{b}$ such that $I\cup\{x\}$ is an independent
set of $M$. Thus $(M,\Omega)$ is not a 2-sheltering matroid --- a contradiction.
\end{proof}

\begin{proposition}
\label{prop:off_diag_symmetric_repr} If a tight 2-sheltering matroid
$Q=(M,\Omega)$ is representable over some field $\mathbb{F}$, then
$r(M)=|\Omega|$. Consequently, $Q$ is strictly representable over $\mathbb{F}$.

Moreover, if the tight 2-sheltering matroid $Q$ is (strictly) representable
over $\mathbb{F}$ and $B\in\mathcal{T}(\Omega)$ is a basis of $M$, then for
every $\mathbb{F}$-standard representation
\[
E=\bordermatrix{
& B & T \cr
& I & A
}
\]
of $M$ with respect to $B$, we have that $A$ is a zero-diagonal $T\times T$
matrix with $T=E-B\in\mathcal{T}(\Omega)$.

In particular, if $\mathbb{F}=GF(2)$ then $A$ is symmetric and zero-diagonal.
\end{proposition}

\begin{proof}
Let $Q=(M,\Omega)$ be a tight 2-sheltering matroid representable over
$\mathbb{F}$. Let $B$ be a basis of $Q$. Hence $B$ is an independent set of
$M$. Thus, $M$ has a $GF(2)$-representation
\[
E=\bordermatrix{
& B & T \cr
& I & A \cr
& 0 & C
}
\]
for some matrices $A$ and $C$ and where $I$ and $0$ are the identity matrix
and zero matrix of suitable size. Let $a\in T$. Let $\omega_{a}$ be the skew
class of $\Omega$ containing $a$. Then the rank of $M[B\setminus\omega_{a}]$
is smaller than the rank of $M[B]$. Since $Q$ is tight, the rank of
$M[(B\setminus\omega_{a})\cup\{a\}]$ is equal to that of $M[B\setminus
\omega_{a}]$. Hence both (1) the nonzero diagonal entry of $A$ at index $a\in
T$ is zero and (2) the column of $C$ corresponding to $a$ is zero.
Consequently, $C$ is the zero matrix and $A$ is zero-diagonal. Since $C$ is
the zero matrix, $r(M)=|\Omega|$.

It follows from Lemma~\ref{lem:symmetric_iff_2sm} that if $\mathbb{F}=GF(2)$,
then $A$ is symmetric.
\end{proof}

In Subsection~\ref{ssec:strongly_binary} we explain that
Proposition~\ref{prop:off_diag_symmetric_repr} for the case $\mathbb{F} =
GF(2)$ is closely related to Property~5.2 of Bouchet and Duchamp
\cite{Bouchet_1991_67}.

We remark that Proposition~\ref{prop:off_diag_symmetric_repr} is also closely related to the following result shown in \cite{bouchet1987} in the context of even delta-matroids (even delta-matroids correspond to tight $2$-matroids by \cite[Theorem~5.3]{B3}). For convenience we also provide a short proof without using delta-matroids.
\begin{proposition}[\cite{bouchet1987}]
Let $Q$ be a 2-sheltering matroid having $\mathbb{F}$-representation
\[
E=\bordermatrix{
& T_1 & T_2 \cr
& I & A
}
\]
with $A$ skew-symmetric. Then $A$ is zero-diagonal if and only if $Q$ is tight.
\end{proposition}
\begin{proof}
The if direction follows from Proposition~\ref{prop:off_diag_symmetric_repr}. Note that for the if direction skew-symmetry is not needed.

For the only-if direction we use the fact that a 2-matroid $Z$ is tight if and only if for any basis $B$ and skew class $\omega$ of $Z$, $B \Delta \omega$ is not a basis (see \cite[Theorem~4.2]{B3}). Let $Z = \mathcal{Z}(Q)$, let $B$ be a basis of $Z$, and $\omega$ be a skew class of $Z$. Assume that skew-symmetric matrix $A$ is zero-diagonal. By applying principal pivot transform to $E$, we have that $Q$ is
represented by
\[
E^{\prime}=\bordermatrix{
& B & T \cr
& I & A'
}
\]
for some zero-diagonal skew-symmetric matrix $A^{\prime}$ and some $T\in\mathcal{T}(\Omega
)$. Let $\omega=\{x,y\}$ with $x\in B$. Since the diagonal entry of $A'$ corresponding to $y$ is zero, there is a circuit $C \subseteq B \Delta \omega$ containing $y$. Hence, $B \Delta \omega$ is not a basis. We conclude that $Z$ is tight, and therefore $Q$ is tight.
\end{proof}


The next example illustrates that not every binary 2-matroid is strictly
binary. Therefore, the condition of tightness in
Proposition~\ref{prop:off_diag_symmetric_repr} is essential.

\begin{example}
Let $Z$ be the $2$-matroid over $(U,\Omega)$, where $U=\{a^{\prime},b^{\prime
},c^{\prime},a,b,c\}$, $\Omega=\{\{a^{\prime},a\},\{b^{\prime},b\},\{c^{\prime
},c\}\}$ and the family of circuits $\mathcal{C}$ of $Z$ is $\{\{a^{\prime
},b^{\prime},c^{\prime}\},\{a,b,c\}\}$. Clearly, $Z$ is sheltered by the
binary matroid $M$ with ground set $U$ and $\mathcal{C}$ as the family of
circuits. The rank of $M$ is $4$. We argue that $M$ is the unique binary
matroid that shelters $Z$. Indeed, since $|U|=6$, a binary matroid $M^{\prime
}$ that shelters $Z$ cannot have the Fano matroid (or its dual), the cocycle
matroid of $K_{3,3}$, or the cocycle matroid of $K_{5}$ (which have ground set
sizes $7$, $9$, and $10$, respectively) as a minor. Hence $M^{\prime}$ is
graphic. It is easy to see that any graphic matroid of ground set size $6$
with two disjoint triangles is isomorphic to $M$; as the ground sets of $M$
and $M^{\prime}$ coincide and the elements of $\mathcal{C}$ are circuits in
both $M$ and $M^{\prime}$, it follows that $M=M^{\prime}$. Since $M$ is the
unique binary matroid that shelters $Z$, there cannot be
a binary matroid of rank $3$ that shelters $Z$. Thus, $Z$ is binary but not
strictly binary.
\end{example}

\subsection{Binary tight 3-matroids and isotropic matroids}

The main results of this subsection are Theorems~\ref{thm:char_bt3sm} and
\ref{thm:3shelt_isotrm} which characterize binary tight $3$-matroids and
isotropic matroids, respectively.

First we need the following result of \cite{BH3}.

\begin{lemma}
[Theorem~13 of \cite{BH3}]\label{lem:unique_3m} Let $\Omega$ be a partition of
some finite set $U$ with for each $\omega\in\Omega$, $|\omega| \geq3$. Let
$T\in\mathcal{T}(\Omega)$. If $Z$ is a multimatroid over $(U\setminus
T,\Omega^{\prime})$ with $\Omega^{\prime}=\{\omega\setminus T \mid\omega
\in\Omega\}$, then there is at most one tight multimatroid $Z^{\prime}$ over
$(U,\Omega)$ with $Z^{\prime}-T=Z$.
\end{lemma}


\begin{theorem}
\label{thm:char_bt3sm} Let $Z=(U,\Omega,\mathcal{I})$ be a $3$-matroid. The
following statements are equivalent.

\begin{enumerate}
\item \label{item:tight_bin} $Z$ is tight and binary.

\item \label{item:tight_strict_bin} $Z$ is tight and strictly binary.

\item \label{item:isotropic} $Z=\mathcal{Z}(Q)$ for some $Q=(M,\Omega)$ where
$M$ can be represented by the matrix
\[
\bordermatrix{
& T_1 & T_2 & T_3\cr
& I & A & A+I
}\text{,}
\]
for some $V\times V$-symmetric matrix $A$ over $GF(2)$ and some transversal
3-tuple $\tau=(T_{1},T_{2},T_{3})$ of $(U,\Omega)$.
\end{enumerate}
\end{theorem}

\begin{proof}
Trivially, Statement~\ref{item:tight_strict_bin} implies
Statement~\ref{item:tight_bin}.

Assume that Statement~\ref{item:isotropic} holds, and let $G$ be the looped
simple graph whose adjacency matrix is $A$. We recall from~\cite{Tnewnew}
that $M[IAS(G)]$ is a tight 3-sheltering matroid with $\Omega$ the partition
of $W(G)$ into vertex triples. Thus $Q=(M,\Omega)$ is a tight $3$-sheltering
matroid. Note that $Q$ is strictly binary since $M$ is binary and $r(M) =
|\Omega|$. Hence Statement~\ref{item:tight_strict_bin} holds.

Assume now that the Statement~\ref{item:tight_bin} holds. Then $Z=\mathcal{Z}%
(Q)$ for some $Q=(M,\Omega)$ such that $M$ is binary, and $Z$ is tight. Let
$T_{1}$ be a basis of $Z$. Let $T_{2}=\{u\in U\mid(T_{1}\setminus\omega
)\cup\{u\}\mbox{ with }u\in\omega\in\Omega\mbox{ is not a basis of }Z\}$.
Since $Z$ is tight, $T_{2}$ is a transversal. Since $T_{1}$ is a basis of $Z$,
$T_{1}$ is an independent set of $Q$. Let
\[
\bordermatrix{
& T_1 & T_2 & T_3 \cr
& I & A & C \cr
& 0 & B & D
}
\]
be a representation of $M$ with respect to $T_{1}$ such that $\tau
=(T_{1},T_{2},T_{3})$ is a transversal 3-tuple of $\Omega$. By the definition
of $T_{2}$, all diagonal entries of $A$ are zero and $B$ is a zero matrix (the
argument is identical to the one given in the proof of
Proposition~\ref{prop:off_diag_symmetric_repr}). Since $Q-T_{3}$ is a
2-sheltering matroid, we have by Lemma~\ref{lem:symmetric_iff_2sm} that $A$ is
symmetric. By applying Lemma~\ref{lem:unique_3m} to 2-matroid $\mathcal{Z}(Q)-T_{3}$, we see there is at most one tight 3-matroid
$Z^{\prime}$ over $(U,\Omega)$ with $Z^{\prime}-T_{3}=\mathcal{Z}(Q)-T_{3}$. The proof that Statement~\ref{item:isotropic} implies Statement~\ref{item:tight_strict_bin} shows that if we take $D$ to be the zero matrix and $C$ to be $A+I$, then this matrix represents a $3$-sheltering matroid $Q'$ with $\mathcal{Z}(Q')$ tight. Moreover, $\mathcal{Z}(Q')-T_{3} = \mathcal{Z}(Q)-T_{3}$. Therefore, $Z = Z^{\prime} = \mathcal{Z}(Q')$, and we notice that $Q'$ is of the form of Statement~\ref{item:isotropic} (the zero rows of the matrix do not influence the matroid $M$). Hence Statement~\ref{item:isotropic} holds.
\end{proof}

While Theorem~\ref{thm:char_bt3sm} shows that every binary tight $3$-matroid $Z$ is equal to $\mathcal{Z}(Q)$ with $Q$ the strictly binary tight $3$-sheltering matroid of the form given by Statement 3, this does not exclude the possible existence of some other
strictly binary tight $3$-sheltering matroid $Q^{\prime}=(M^{\prime},\Omega)$
with $\mathcal{Z}(Q)=\mathcal{Z}(Q^{\prime})$. Indeed, for the distinct strictly binary tight $3$-sheltering matroids $Q_1 = (M_1,\Omega)$ and $Q_2 = (M_2,\Omega)$, where $\Omega = \{ \{a_1,a_2,a_3\}, \{b_1,b_2,b_3\} \}$ and the matroids $M_1$ and $M_2$ are represented by
$$
\bordermatrix{
& a_1 & b_1 & a_2 & b_2 & a_3 & b_3 \cr
& 1   & 0   & 0   & 0   & 1   & 0 \cr
& 0   & 1   & 0   & 0   & 0   & 1
}
\mbox{ and }
\bordermatrix{
& a_1 & b_1 & a_2 & b_2 & a_3 & b_3 \cr
& 1   & 0   & 0   & 0   & 1   & 0 \cr
& 0   & 1   & 0   & 0   & 1   & 1
},
$$ respectively, we have $\mathcal{Z}(Q_1)=\mathcal{Z}(Q_2)$. The next result shows that this
cannot happen if each $\omega\in\Omega$ is an element of the cycle space of $M^{\prime}$. This
result characterizes isotropic matroids.

\begin{theorem}
\label{thm:3shelt_isotrm} Let $Q=(M,\Omega)$ be a $3$-sheltering matroid. The
following statements are equivalent.

\begin{enumerate}
\item $Q$ is strictly binary and each $\omega\in\Omega$ is an element of the cycle space of $M$.

\item $M$ is isomorphic to some isotropic matroid where $\Omega$ is the set of
vertex triples.
\end{enumerate}
\end{theorem}

\begin{proof}
Assume the second statement holds. Recall that for isotropic matroids each
vertex triple is an element of the cycle space. Also, if $M$ is isomorphic to some isotropic
matroid, then $Q$ is obviously strictly binary.

Conversely, assume the first statement holds. Since $Q$ is strictly binary and
$\mathcal{Z}(Q)$ is nondegenerate, $M$ is of rank $|\Omega|$ and contains a
basis $T_{1}$ that is a subtransversal. Let
\[
\bordermatrix{
& T_1 & T_2 & T_3 \cr
& I & A & B
}
\]
be a standard representation of $M$ with respect to $T_{1}$ such that $\tau=
(T_{1},T_{2},T_{3})$ is a transversal 3-tuple of $\Omega$. Since each
$\omega\in\Omega$ is an element of the cycle space of $M$, the columns belonging to each $\omega
\in\Omega$ sum to $0$ and so we have $B = A+I$. By swapping elements from
$T_{2}$ and $T_{3}$, we may assume, without loss of generality, that each
diagonal entry of $A$ is zero. By Lemma~\ref{lem:symmetric_iff_2sm}, $A$ is
symmetric since $Q-T_{3}$ is a 2-sheltering matroid. Hence $M$ is isomorphic
to some isotropic matroid.
\end{proof}

In other words, if $Q=(M,\Omega)$ is a $3$-sheltering matroid where $M$ is
binary and of rank $|\Omega|$, and each $\omega\in\Omega$ is an element of the cycle space of $M$,
then $M$ is isomorphic to some isotropic matroid (where $\Omega$ is the set of
vertex triples).

Note that if $Q$ is isomorphic to some isotropic matroid, then $Q$ is tight.
Hence, by Theorem~\ref{thm:3shelt_isotrm}, if $Q$ is strictly binary and each
$\omega\in\Omega$ is an element of the cycle space of $M$, then $Q$ is tight.

\subsection{Strongly representable 2-matroids}

\label{ssec:strongly_binary} In this subsection we consider a version of
representability for 2-matroids that is stronger than representability. This
stronger version corresponds to the definition of representability of
delta-matroids from Bouchet \cite{bouchet1987}. However, this definition does
not seem to extend naturally to multimatroids other than 2-matroids.

We say that a $2$-sheltering matroid $Q$ is \emph{strongly representable} over
some field $\mathbb{F}$ if $Q=\mathcal{Q}(A,\tau,2)$ for some skew-symmetric
matrix $A$ over $\mathbb{F}$ and some transversal $2$-tuple $\tau$. We say
that a 2-matroid $Z$ is \emph{strongly representable} over $\mathbb{F}$ if
there is a $2$-sheltering matroid $Q$ strongly representable over $\mathbb{F}$
such that $\mathcal{Z}({Q})=Z$. We say that $Q$ ($Z$, respectively) is \emph{strongly
binary} if $Q$ ($Z$, respectively) is strongly representable over $GF(2)$. If
$Q=(M,\Omega)$ is strongly representable over $\mathbb{F}$, then $Q$ is
certainly strictly representable over $\mathbb{F}$. By
Proposition~\ref{prop:off_diag_symmetric_repr}, the converse holds in case $Q$
is strictly binary and tight. Consequently, a tight 2-sheltering matroid is
strictly binary if and only if it is strongly binary.

We should mention that Proposition~\ref{prop:off_diag_symmetric_repr} is
closely related to Property~5.2 of Bouchet and Duchamp \cite{Bouchet_1991_67}:
if an even delta-matroid is weakly binary, then it is binary. The relationship
between the results arises from two facts: if a 2-matroid is strictly binary
by our definition, then the associated delta-matroid is weakly binary by their
definition; and a binary delta-matroid is even if and only if the associated
2-matroid is tight. Like the property of Bouchet and Duchamp,
Proposition~\ref{prop:off_diag_symmetric_repr} does not hold for strictly
binary 2-sheltering matroids in general. In fact, their example $S_{2}$ gives
us the following example of a strictly binary 2-sheltering matroid that
necessarily requires that $A$ be asymmetric.

\begin{example}
Let
\[
E=\bordermatrix{
& a_1 & b_1 & c_1 & a_2 & b_2 & c_2 \cr
a & 1 & 0 & 0 & 1 & 0 & 1 \cr
b & 0 & 1 & 0 & 1 & 1 & 0 \cr
c & 0 & 0 & 1 & 0 & 1 & 1
},
\]
and $B=\{a_{1},b_{1},c_{1}\}$. Then one may verify that the binary matroid $M$
represented by $E$, with the partition $\Omega=\{\{a_{1},a_{2}\}$,
$\{b_{1},b_{2}\}$, $\{c_{1},c_{2}\}\}$, forms a strictly binary 2-sheltering
matroid $Q=(M,\Omega)$. However, Bouchet and Duchamp \cite{Bouchet_1991_67}
observe, in the context of delta-matroids, that $Q$ is not strongly binary.
\end{example}

The interested reader can verify the observation of Bouchet and Duchamp that
$Q$ is not strongly binary in three steps, as follows. First, find all the
transversals of $W(G)$ that are bases of $M$; there are seven, including $B$ and (for instance)
$B^{\prime}=\{a_{2},b_{2},c_{1}\}$. Second, for each of the six bases other than $B$, find the fundamental circuits of the remaining elements.
For instance, the fundamental circuits with respect to $B^{\prime}$ are
$C(a_{1},B^{\prime})=\{a_{1},a_{2},b_{2},c_{1}\}$, $C(b_{1},B^{\prime
})=\{b_{1},b_{2},c_{1}\}$ and $C(c_{2},B^{\prime})=\{c_{2},a_{2},b_{2}\}$. The
representation of $M$ corresponding to a basis $B\in\mathcal{T}(G)$ is a $GF(2)$-matrix
of the form $%
\begin{pmatrix}
I & A
\end{pmatrix}
$, where the columns of $A$ are the incidence vectors of the fundamental
circuits. The third step is to verify that none of these $A$ matrices is
symmetric. For instance, the $A$ matrix corresponding to $B^{\prime}$ is not
symmetric because $b_{2}\in C(a_{1},B^{\prime})$ and $a_{2}\notin
C(b_{1},B^{\prime})$.

We now show that every strongly binary 2-matroid can be sheltered by exactly
one strongly binary sheltering matroid.

\begin{proposition}
\label{lem:unique_bin_2sh_dm} For every strongly binary $2$-matroid $Z$, there
is a unique strongly binary $2$-sheltering matroid $Q$ such that
$\mathcal{Z}({Q})=Z$.
\end{proposition}

\begin{proof}
Let $\mathcal{Z}(\mathcal{Q}(A,\tau,2))=\mathcal{Z}(\mathcal{Q}(A^{\prime
},\tau^{\prime},2))$ for some symmetric matrices $A$ and $A^{\prime}$ over
$GF(2)$. By applying principal pivot transform, we have $\mathcal{Q}%
(A^{\prime},\tau^{\prime},2)=\mathcal{Q}(A^{\prime\prime},\tau,2)$ for some
symmetric matrix $A^{\prime\prime}$.

Let $\tau=(T_{1},T_{2})$. Assume $A\neq A^{\prime\prime}$. If there is an
$x\in T_{2}$ such that $A[\{x\}]\neq A^{\prime\prime}[\{x\}]$, then
$T_{1}\Delta p$ with $p$ the skew pair containing $x$ is a basis of exactly
one of $\mathcal{Q}(A,\tau,2)$ and $\mathcal{Q}(A^{\prime\prime},\tau,2)$ ---
a contradiction since $\mathcal{Z}(\mathcal{Q}(A,\tau,2))=\mathcal{Z}%
(\mathcal{Q}(A^{\prime\prime},\tau^{\prime},2))$. Consequently, $A$ and
$A^{\prime\prime}$ coincide on the diagonal entries, and so $A$ and
$A^{\prime\prime}$ must differ on some off-diagonal entry. Thus there are
$x,y\in T_{2}$ such that
\[
A[\{x,y\}]=\bordermatrix{
& x & y \cr
x & a & 1 \cr
y & 1 & b
}\text{ and }A^{\prime}[\{x,y\}]=\bordermatrix{
& x & y \cr
x & a & 0 \cr
y & 0 & b
}
\]
for some $a,b\in\{0,1\}$ (or the roles of $A$ and $A^{\prime}$ are reversed).
Now, $A[\{x,y\}]$ is singular if and only if $A^{\prime}[\{x,y\}]$ is not
singular. Hence $T_{1}\Delta(p\cup q)$, with $p$ and $q$ the skew pairs
containing $x$ and $y$, is a basis of exactly one of $\mathcal{Q}(A,\tau,2)$
and $\mathcal{Q}(A^{\prime\prime},\tau,2)$ --- a contradiction. Therefore,
$A=A^{\prime\prime}$.
\end{proof}


\section{Isomorphisms of isotropic matroids} \label{sec:isom_isotropic_mat}
In this section we discuss the connection between local equivalence and isomorphisms of the isotropic matroids $M[IAS(G)]$, and the connection between pivot equivalence and isomorphisms of the restricted isotropic matroids $M[IA(G)]$. In the third subsection we mention the surprising observation that every looped simple graph $G$ has an associated bipartite simple graph $B(G)$, such that $G$ and $H$ are locally equivalent if and only if $B(G)$ and $B(H)$ are pivot equivalent.

\subsection{Compatible and non-compatible isomorphisms}
We begin by summarizing the way local complementations induce isomorphisms of iso\-tropic matroids. A full account is given in~\cite{Tnewnew}.

\begin{proposition}
\label{iso3}(\cite{Tnewnew}) If $G$ is a looped simple graph with a vertex $v$
then there are induced isomorphisms $\beta_{\ell}^{v}:M[IAS(G)]\rightarrow
M[IAS(G_{\ell}^{v})]$, $\beta_{ns}^{v}:M[IAS(G)]\rightarrow M[IAS(G_{ns}%
^{v})]$ and $\beta_{s}^{v}:M[IAS(G)]\rightarrow M[IAS(G_{s}^{v})] $. These
isomorphisms have $\beta_{\ast}^{v}(\alpha_{G}(x))=\alpha_{G_{\ast}^{v}}(x)$
for all $\alpha\in\{\phi,\chi,\psi\} $ and $x\in V(G)$, except as follows:

\begin{enumerate}
\item $\beta_{\ell}^{v}(\chi_{G}(v))=\psi_{G_{\ell}^{v}}(v)$ and $\beta_{\ell
}^{v}(\psi_{G}(v))=\chi_{G_{\ell}^{v}}(v)$.

\item If $v$ is not looped then $\beta_{ns}^{v}(\phi_{G}(v))=\psi_{G_{ns}^{v}%
}(v)$ and $\beta_{ns}^{v}(\psi_{G}(v))=\phi_{G_{ns}^{v}}(v)$.

\item If $v$ is looped then $\beta_{ns}^{v}(\phi_{G}(v))=\chi_{G_{ns}^{v}}(v)
$ and $\beta_{ns}^{v}(\chi_{G}(v))=\phi_{G_{ns}^{v}}(v)$.

\item If $v$ is not looped then $\beta_{s}^{v}(\phi_{G}(v))=\psi_{G_{s}^{v}%
}(v)$ and $\beta_{s}^{v}(\psi_{G}(v))=\phi_{G_{s}^{v}}(v)$; also if $w\in
N(v)$ then $\beta_{s}^{v}(\chi_{G}(w))=\psi_{G_{s}^{v}}(w)$ and $\beta_{s}%
^{v}(\psi_{G}(w))=\chi_{G_{s}^{v}}(w)$.

\item If $v$ is looped then $\beta_{s}^{v}(\phi_{G}(v))=\chi_{G_{s}^{v}}(v)$
and $\beta_{s}^{v}(\chi_{G}(v))=\phi_{G_{s}^{v}}(v)$; also if $w\in N(v)$ then
$\beta_{s}^{v}(\chi_{G}(w))=\psi_{G_{s}^{v}}(w)$ and $\beta_{s}^{v}(\psi
_{G}(w))=\chi_{G_{s}^{v}}(w)$.
\end{enumerate}
\end{proposition}

\begin{proof}
For $G_{\ell}^{v}$ the assertion is obvious, as the only difference
between $IAS(G)$ and $IAS(G_{\ell}^{v})$ is that the $\chi(v)$ and $\psi(v)$ columns are
transposed. 

For $G_{ns}^{v}$ the situation is a little more complicated. If $v$ is not looped, $IAS(G_{ns}^{v})$ can be obtained from $IAS(G)$ by interchanging the $\phi_{G}(v)$ and $\psi_{G}(v)$ columns, and then adding the $v$ row to every other row corresponding to a neighbor of $v$. Elementary row operations do not affect the matroid represented by a matrix, of course, so there is an isomorphism $\beta:M[IAS(G)]\rightarrow M[IAS(G_{ns}^{v})]$ that is given by $\beta(\phi_{G}(v))=\psi_{G_{ns}^{v}}(v)$, $\beta(\psi_{G}(v))=\phi_{G_{ns}^{v}}(v)$, and otherwise $\beta(\alpha_{G}(w))=\alpha_{G_{ns}^{v}}(w)$ $\forall\alpha\in\{\phi,\chi,\psi\}$ $\forall w\in V(G)=V(G_{ns}^{v})$.

The remaining assertions follow, using compositions of loop complementations and non-simple local complementations at unlooped vertices.
\end{proof}

\begin{corollary}
\label{induced}
If two looped simple graphs $G$ and $H$ are locally equivalent (up to isomorphism), then there is an isomorphism $\beta:M[IAS(G)]\to M[IAS(H)]$, which maps vertex triples to vertex triples.
\end{corollary}

\begin{proof}
Up to isomorphism, $H$ can be obtained from $G$ through a sequence of individual local complementations and loop complementations. A matroid isomorphism of the type mentioned in the statement is the composition of the isomorphisms induced by the individual local complementations and loop complementations.
\end{proof}

We say that an isomorphism of this type is \emph{compatible} with the partitions of $W(G)$ and $W(H)$ into vertex triples, or that it is \emph{induced} by a sequence of loop and local complementations used to obtain an isomorph of $H$ from $G$. Notice that there is an associated bijection between $V(G)$ and $V(H)$ such that for each $v\in V(G)$, $\beta$ maps the vertex triple $\tau_{G}(v)$ to the vertex triple $\tau_{H}(\beta(v))$. In the special cases mentioned in Proposition \ref{iso3}, this vertex bijection does not appear explicitly because it is the identity map of $V(G)$. In general, we use $\beta$ to denote both a compatible isomorphism of isotropic matroids and the associated vertex bijection; there is little danger of confusing the two, because of the difference between their domains.

It turns out that the converse of Corollary~\ref{induced} is also valid; details are discussed in~\cite{Tnewnew}. The discussion of~\cite{Tnewnew} also yields a characterization of local equivalence without allowing for graph isomorphisms: 

\begin{proposition}
$G$ and $H$ are locally equivalent if and only if there is a compatible isomorphism $\beta:M[IAS(G)]\to M[IAS(H)]$ whose associated bijection $V(G) \to V(H)$ is the identity map of $V(G)=V(H)$.
\end{proposition}

The most difficult result of \cite{Tnewnew} is this.

\begin{proposition}
\label{isos}(\cite{Tnewnew}) If there is an isomorphism between the matroids $M[IAS(G)]$ and $M[IAS(H)]$, then $G$ and $H$ are locally equivalent (up to graph isomorphism).
\end{proposition}

What makes Proposition \ref{isos} difficult is the fact that unlike a compatible isomorphism, an arbitrary matroid isomorphism between $M[IAS(G)]$ and $M[IAS(H)]$ need not be
directly connected with any particular sequence of loop and local complementations that relates $G$ to $H$. Proposition \ref{isos} is proven in \cite{Tnewnew} by showing that an arbitrary matroid isomorphism may be
incrementally \textquotedblleft deformed\textquotedblright\ into a
compatible isomorphism. The process involves two types of incremental
deformations; one type is focused on a pair of vertices and the other type
is focused on a set of four vertices. The second type of deformation does
not yield a precise correspondence between the four vertices to
which the deformation is applied and the four vertices that result from the
deformation.

We do not provide details of this deformation process here, but we take a
moment to discuss an example. Let $C_{5}$ be the graph with vertices denoted
1, 2, 3, 4 and 5, which form a 5-cycle in the given order; that is, 1 is
adjacent to 5 and 2, 2 is adjacent to 1 and 3, 3 is adjacent to 2 and 4, and
4 is adjacent to 3 and 5. Let $H$ be the graph with $V(H)=\{a,b,c,d,e\}$,
such that $a,b,c,d,e$ form a 5-cycle in the given order, and there is also
an edge connecting $c$ and $e$. See Figure~\ref{circh1f3}.

\begin{figure}[ptb]%
\centering
\includegraphics[
trim=2.408843in 8.167255in 2.007794in 1.071946in,
natheight=11.005600in,
natwidth=8.496800in,
height=1.3526in,
width=3.0874in
]%
{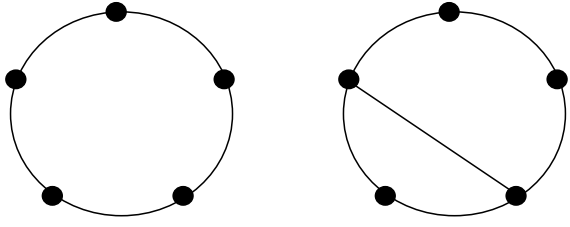}%
\caption{$C_5$ and $H$.}%
\label{circh1f3}%
\end{figure}

Here is the matrix $IAS(C_{5})$, with the columns listed in an unusual order.
\[
\bordermatrix{
& \phi _{1} & \psi _{1} & \chi _{1} & \psi _{2} & \chi _{5} & \psi _{3} & 
\phi _{4} & \chi _{3} & \phi _{2} & \psi _{5} & \psi _{4} & \chi _{2} & \chi
_{4} & \phi _{5} & \phi _{3} \cr
1 & 1 & 1 & 0 & 1 & 1 & 0 & 0 & 0 & 0 & 1 & 0 & 1 & 0 & 0 & 0 \cr 
2 & 0 & 1 & 1 & 1 & 0 & 1 & 0 & 1 & 1 & 0 & 0 & 0 & 0 & 0 & 0 \cr
3 & 0 & 0 & 0 & 1 & 0 & 1 & 0 & 0 & 0 & 0 & 1 & 1 & 1 & 0 & 1 \cr
4 & 0 & 0 & 0 & 0 & 1 & 1 & 1 & 1 & 0 & 1 & 1 & 0 & 0 & 0 & 0 \cr 
5 & 0 & 1 & 1 & 0 & 0 & 0 & 0 & 0 & 0 & 1 & 1 & 0 & 1 & 1 & 0 \cr
}
\]

Here is the matrix $IAS(H)$, with the columns grouped according to the vertex triples.
\[
\bordermatrix{
& \phi _{a} & \chi _{a} & \psi _{a} & \phi _{b} & \chi _{b} & \psi _{b} & 
\phi _{c} & \chi _{c} & \psi _{c} & \phi _{d} & \chi _{d} & \psi _{d} & \phi
_{e} & \chi _{e} & \psi _{e} \cr
a & 1 & 0 & 1 & 0 & 1 & 1 & 0 & 0 & 0 & 0 & 0 & 0 & 0 & 1 & 1 \cr 
b & 0 & 1 & 1 & 1 & 0 & 1 & 0 & 1 & 1 & 0 & 0 & 0 & 0 & 0 & 0 \cr
c & 0 & 0 & 0 & 0 & 1 & 1 & 1 & 0 & 1 & 0 & 1 & 1 & 0 & 1 & 1 \cr
d & 0 & 0 & 0 & 0 & 0 & 0 & 0 & 1 & 1 & 1 & 0 & 1 & 0 & 1 & 1 \cr 
e & 0 & 1 & 1 & 0 & 0 & 0 & 0 & 1 & 1 & 0 & 1 & 1 & 1 & 0 & 1 \cr
}
\]

We assert that these two matrices represent isomorphic binary matroids, with
an isomorphism matching matroid elements according to the given orders of
the columns. This assertion is verified by checking that the alleged
isomorphism matches sets of columns that sum to 0. We do not verify all of these matches but here are three instances: $\{\chi _{4},\phi _{5},\phi _{3}\}
$ and $\{\phi _{e},\chi _{e},\psi _{e}\}$ both sum to 0; $\{\psi _{2},\chi
_{3},\psi _{5},\chi _{4}\}$ and $\{\phi _{b},\chi _{c},\phi _{d},\phi _{e}\}$
both sum to 0; and $\{\psi _{2},\phi _{4},\phi _{2},\psi _{5},\chi _{4}\}$
and $\{\phi _{b},\phi _{c},\psi _{c},\phi _{d},\phi _{e}\}$ both sum to 0.

The given isomorphism $M[IAS(C_{5})]\cong M[IAS(H)]$ matches the vertices 1 and $a$ to each other directly, but there is no such direct matching of the other vertices. Instead, the elements of the vertex triples corresponding to 2, 3, 4 and 5 are
rearranged in a complicated way to produce the vertex triples
corresponding to $b,c,d$ and $e$. Proposition \ref{isos} is satisfied
as $C_{5}$ and $H$ are locally equivalent up to isomorphism --- local complementation of $C_{5}$ with respect to any vertex yields a graph $G$ isomorphic to $H$ --- but the given isomorphism $M[IAS(C_{5})]\cong M[IAS(H)]$ is not directly connected to any isomorphism
between $H$ and a graph locally equivalent to $C_{5}$.

\subsection{Transverse circuits and transverse matroids}

For isotropic matroids, we have the following.

\begin{theorem}
\label{theory}Let $G$ and $H$ be looped simple graphs. Then any one of the
following conditions implies the others:

\begin{enumerate}
\item $G$ and $H$ are locally equivalent, up to isomorphism.

\item There is a compatible isomorphism between the isotropic matroids of $G$ and $H$.

\item There is an isomorphism between the isotropic matroids of $G$ and $H$.

\item There is a bijection between $W(G)$ and $W(H)$, which defines
isomorphisms between the transverse matroids of $G$ and those of $H$.

\item There is a bijection between $W(G)$ and $W(H)$, under which vertex
triples and transverse circuits of $G$ and $H$ correspond.
\end{enumerate}
\end{theorem}

\begin{proof}
We begin with the implication $5\Rightarrow3$. First we recall that a binary matroid is uniquely determined by its cycle space (the span of its circuits under symmetric difference) along with its ground set \cite{Ghouila}. Note that every vertex triple is an element of the cycle space of the matroid $M[IAS(G)]$. To verify the implication, it suffices to show that the cycle space of $M[IAS(G)]$ is generated by the vertex triples and the cycle spaces of the transverse matroids. Let $C$ be an element of the cycle space of $M[IAS(G)]$, and let $O$ be the set of vertex triples $\omega$ with $|C\cap\omega|>1$. Then $C^{\prime} =C\mathop{\mathrm{\Delta}}(\Delta_{\omega\in O}\omega)$ is an element of the cycle space, with $|C^{\prime}\cap\omega|\leq1$ for every vertex triple, so $C^{\prime}$ is included in the cycle space of some (in fact, every) transverse matroid which has $C^{\prime}$ as a
subset of its ground set. Hence $C=C^{\prime}\mathop{\mathrm{\Delta}}(\Delta_{\omega\in O}\omega)$ is a sum of transverse circuits and vertex triples.

The implications $3\Rightarrow 2\Rightarrow 1$ are verified in \cite{Tnewnew}, the implication $1\Rightarrow2$ is verified in Corollary~\ref{induced}, and the implication $2\Rightarrow4$ is obvious. The implication $4\Rightarrow5$ is almost obvious; we need only observe that if transverse matroids correspond under a bijection between $W(G)$ and $W(H)$, then vertex triples must correspond too.
\end{proof}

\subsection{Local equivalence from pivot equivalence}

Here is another equivalence relation often mentioned in conjunction with local equivalence.

\begin{definition}
\label{edgepivot}Suppose $vw$ is a nonloop edge of a simple
graph $G$. Then the \emph{edge pivot} of $G$ with respect to $vw$ is
\[
G^{vw}=((G_{s}^{v})_{s}^{w})_{s}^{v}=((G_{s}^{w})_{s}^{v})_{s}^{w}.
\]
If $H$ can be obtained from $G$ using edge pivots then $G$ and $H$ are \emph{pivot equivalent}.
\end{definition}

It is not obvious at first glance that the two triple local complements mentioned in Definition~\ref{edgepivot} are indeed equal, but the reader who has not seen the equality before will have no trouble verifying it. We should mention that ``pivot equivalence'' is often defined in a more complicated way for looped simple graphs: non-simple local complementations at looped vertices are allowed, and edge pivots involving looped vertices are disallowed. The details are not important in this paper, though, because we discuss pivot equivalence only in this subsection and the next, and only for simple graphs.

It is obvious that if $G$ and $H$ are pivot equivalent graphs then they are also locally equivalent. Moreover, the converse is false; for instance, the complete graph $K_{n}$ is not pivot equivalent to any nonisomorphic graph, but it is locally equivalent to a star graph. These observations indicate the well-known fact that pivot equivalence is a strictly finer relation on simple graphs than local equivalence. A surprising consequence of Theorem~\ref{theory} is that despite this fact, local equivalence is indirectly determined by pivot equivalence. The following notion will be useful in explaining this consequence.

\begin{definition}
Let $B$ be a basis of a matroid $M$, with ground set $W$. Then the \emph{fundamental graph} $G_{B}(M)$ of $M$ with respect to $B$ is a bipartite simple graph with vertex classes $B$ and $W-B$, which has an edge connecting $b\in B$ to $w\notin B$ if and only if $b$ is included in the unique circuit $C(w,B)\subseteq B\cup \{w\}$.
\end{definition}

\begin{proposition}
(\cite[Prop. 4.3.2]{O})
\label{funconn}
Let $B$ be a basis of a matroid $M$. Then $M$ is a connected matroid if and only if $G_{B}(M)$ is a connected graph.
\end{proposition}

If $M$ is a disconnected matroid then $M$ is a direct sum $\oplus N_i$ of connected matroids, and a fundamental graph $G_{B}(M)$ is a disjoint union of fundamental graphs $G_{B}(N_i)$. In particular, an isolated vertex of $G_{B}(M)$ corresponds to either a loop or a coloop of $M$. 

\begin{proposition}
\label{fundamental}
Let $M_1$ and $M_2$ be connected binary matroids with bases $B_1$ and $B_2$, respectively. Then $G_{B_1}(M_1)$ and $G_{B_2}(M_2)$ are pivot equivalent if and only if $M_{1} = M_{2}$ or $M_{1} = M_{2}^{\ast}$.
\end{proposition}

Proposition~\ref{fundamental} is certainly implicit in the discussion of matroid representations in Oxley's book~\cite{O}, though it is not explicitly stated there. An explicit statement equivalent to Proposition~\ref{fundamental} is proven by Oum~\cite[Corollary 3.5]{Oum}. (Verifying the equivalence between these statements requires the elementary observation that the vertex classes of a connected bipartite graph are unique.)

\begin{definition}
For a looped simple graph $G$, we denote by $B(G)$ the fundamental graph of $M[IAS(G)]$ with respect to the basis $\Phi(G) = \{\phi_{G}(v)\mid v\in V(G)\}$.
\end{definition}

If $v\in V(G)$, then the neighborhood of $\chi_{G}(v)$ (respectively $\psi_{G}(v)$) in $B(G)$ gives a set of $\phi$ columns of $IAS(G)$ whose sum is equal to the $\chi_{G}(v)$ column (respectively the $\psi_{G}(v)$ column). Hence $B(G)$ is the bipartite graph with adjacency matrix
\[
A(B(G))=%
\begin{pmatrix}
0 & A(G) & A(G)+I\\
A(G) & 0 & 0\\
A(G)+I & 0 & 0
\end{pmatrix}
\text{.}
\]

\begin{corollary}
\label{localpivot}Two looped simple graphs $G$ and $H$ are locally equivalent (up to isomorphism) if and only if $B(G)$ and $B(H)$ are pivot equivalent (up to isomorphism).
\end{corollary}

\begin{proof}
As noted above, an isolated vertex of $B(G)$ corresponds to either a loop or a coloop of $M[IAS(G)]$. As observed in Propositions~\ref{serieschar} and~\ref{loopchar}, $M[IAS(G)]$ has no coloop, and a loop in $M[IAS(G)]$ corresponds to an isolated vertex of $G$. It follows that $B(G)$ has an isolated vertex if and only if $G$ has an isolated vertex. Isolated vertices are preserved under local equivalence and pivot equivalence, so if either $G$ and $H$ are locally equivalent (up to isomorphism) or $B(G)$ and $B(H)$ are pivot equivalent (up to isomorphism), then $G$ has an isolated vertex if and only if $H$ has an isolated vertex. Induction on $|V(G)|$ allows us to assume that neither $G$ nor $H$ has an isolated vertex; in particular, each of $G,H$ has at least two vertices.

If $G$ and $H$ are locally equivalent (up to isomorphism), then $|V(G)|=|V(H)|$. Also, if $B(G)$ and $B(H)$ are pivot equivalent (up to isomorphism) then $3 \cdot |V(G)|=|V(B(G))|=|V(B(H))|=3 \cdot |V(H)|$. Consequently we may assume that $|V(G)|=|V(H)|>1$.

Suppose for the moment that $G$ is connected; then $M[IAS(G)]$ is connected~\cite[Section 7]{Tnewnew}. If $G$ and $H$ are locally equivalent (up to isomorphism) then Theorem~\ref{theory} tells us that $M[IAS(G)] \cong M[IAS(H)]$, so Proposition~\ref{fundamental} implies that $B(G)$ and $B(H)$ are pivot equivalent (up to isomorphism). For the converse, suppose $B(G)$ and $B(H)$ are pivot equivalent (up to isomorphism). Then Proposition~\ref{fundamental} tells us that $M[IAS(G)] \cong M[IAS(H)]$ or $M[IAS(G)] \cong M[IAS(H)]^{*}$. The latter is impossible, as the rank of $M[IAS(G)]$ is $|V(G)|=|V(H)|$ and the rank of $M[IAS(H)]^{*}$ is $2 \cdot |V(H)|$. Theorem~\ref{theory} tells us that $M[IAS(G)] \cong M[IAS(H)]$ implies $G$ and $H$ are locally equivalent (up to isomorphism).

If $G$ is not connected, let $G_1,\dots,G_c$ be its connected components. Then $M[IAS(G)]$ is the direct sum of the isotropic matroids of $G_1,\dots,G_c$~\cite[Section 7]{Tnewnew}, so $B(G)$ is the disjoint union of $B(G_1),\dots,B(G_c)$. The assertion of the corollary follows from the arguments above, along with the fact that local equivalence and pivot equivalence do not alter the vertex sets of connected components. 
\end{proof}

Before proceeding we remark on an unfortunate clash of nomenclature. Bouchet calls $G$ a ``fundamental graph'' of the isotropic system associated with $G$. See \cite{Bi2, B3} for instance. So, although isotropic systems and isotropic matroids are equivalent (cf.\ Theorem~\ref{iso1}), their notions of fundamental graphs differ.

\subsection{Pivot equivalence and $M[IA(G)]$}

If $G$ is a graph with adjacency matrix $A(G)$ let $IA(G)=\begin{pmatrix} I & A(G) \end{pmatrix}$, and let $M[IA(G)]$ be the binary matroid represented by $IA(G)$. We call $M[IA(G)]$ the \emph{restricted} isotropic matroid of $G$. The ground set of $M[IA(G)]$ is $\{\phi_{G}(v),\chi_{G}(v)\mid v\in V(G)\}\subseteq W(G)$; there is a natural partition of its elements into pairs corresponding to the vertices of $G$.
\begin{proposition}
$M[IA(G)]$ is a 2-sheltering matroid with respect to the natural partition of its ground set. Also, $M[IA(G)] \cong M[IA(G)]^{\ast}$.
\end{proposition}
\begin{proof}
The first assertion follows from the fact that $M[IAS(G)]$ is a 3-sheltering matroid with respect to the partition of $W(G)$ into vertex triples. The second assertion follows from the fact that $A(G)$ is symmetric, cf.\ \cite[Theorem 2.2.8]{O}.
\end{proof}

Theorem~\ref{theory} tells us that there is a direct relationship between local equivalence and isotropic matroids: two graphs are locally equivalent (up to isomorphism) if and only if their isotropic matroids are isomorphic. The following proposition indicates that there is an analogous, but more complicated relationship between pivot equivalence and $M[IA]$ matroids.
\begin{proposition} Let $G$ and $H$ be simple graphs.
\label{pivot2}
\begin{enumerate}
\item $G$ and $H$ are pivot equivalent (up to isomorphism) if and only if $M[IA(G)]$ and $M[IA(H)]$ are isomorphic as 2-sheltering matroids.
\item If $G$ and $H$ are pivot equivalent (up to isomorphism) then $M[IA(G)]$ and $M[IA(H)]$ are isomorphic as matroids. However $M[IA(G)]$ and $M[IA(H)]$ may be isomorphic even if $G$ and $H$ are not pivot equivalent (up to isomorphism).
\item If $G$ and $H$ are bipartite, then $G$ and $H$ are pivot equivalent (up to isomorphism) if and only if $M[IA(G)]$ and $M[IA(H)]$ are isomorphic as matroids.
\end{enumerate}
\end{proposition}

We mention Proposition~\ref{pivot2} to clarify the significance of the preceding subsection, and to contrast with Theorem~\ref{theory}; the results mentioned in Proposition~\ref{pivot2} are known, though perhaps not easy to recognize. Item 1 may be translated from 2-sheltering matroids through 2-matroids to delta-matroids. The translation is ``$G$ and $H$ are pivot equivalent (up to isomorphism) if and only if the corresponding binary delta-matroids are twist equivalent,'' and this assertion follows from the definition of twist equivalence and Geelen's observation~\cite{G} that edge pivots are the only elementary pivots available for the binary delta-matroid of a simple graph. To recognize item 3, note that if $G$ is bipartite and $M$ is a binary matroid with fundamental graph $G$ then 
\[
M[IA(G)]\cong 
M \left[ \begin{pmatrix}
I_1 & 0 &0 & A_1\\
0 & I_2 & A_{1}^{\text{T}} &0\\
\end{pmatrix} \right] 
\cong M[\begin{pmatrix}I_1 & A_1 \end{pmatrix}] \oplus M[\begin{pmatrix}I_2 & A_{1}^{\text{T}} \end{pmatrix}]
\cong M \oplus M^{\ast}
\text{,}
\]
where $I_1$ and $I_2$ are identity matrices. Consequently a fundamental graph of $M[IA(G)]$ consists of two disjoint copies of $G$, so item 3 follows from Proposition~\ref{fundamental}. The positive assertion of item 2 is easy to verify by comparing the matrices $IA(G)$ and $IA(G^{vw})$; it was mentioned in~\cite{Tnewnew}. 

Here is an example to illustrate the negative assertion of item 2. Let $C_6$ be the cycle graph of order 6, with the conventional vertex order -- vertex 1 is adjacent to vertices 6 and 2, vertex 2 is adjacent to vertices 1 and 3, etc. Let $2C_3$ be the disconnected graph consisting of two disjoint copies of $C_3$, with an unconventional vertex order: vertices 1, 3 and 5 lie on one 3-cycle, and vertices 2, 4 and 6 lie on the other. Then
\[
A(C_6)=
\begin{pmatrix}
0 & 1 & 0 & 0 & 0& 1 \\
1 & 0 & 1 & 0 & 0& 0 \\
0 & 1 & 0 & 1 & 0& 0 \\
0 & 0 & 1 & 0 & 1& 0 \\
0 & 0 & 0 & 1 & 0& 1 \\
1 & 0 & 0 & 0 & 1& 0 \\
\end{pmatrix}
 \text{and }
A(2C_3)=
\begin{pmatrix}
0 & 0 & 1 & 0 & 1& 0 \\
0 & 0 & 0 & 1 & 0& 1 \\
1 & 0 & 0 & 0 & 1& 0 \\
0 & 1 & 0 & 0 & 0& 1 \\
1 & 0 & 1 & 0 & 0& 0 \\
0 & 1 & 0 & 1 & 0& 0 \\
\end{pmatrix}
\text{.}
\]
Notice that the two matrices have the same columns. Then $IA(C_6)$ and $IA(2C_3)$ also have the same columns, so $M[IA(C_6)]$ and $M[IA(2C_3)]$ are isomorphic. However $C_6$ and $2C_3$ are not pivot equivalent (up to isomorphism); they are not even locally equivalent (up to isomorphism), as local equivalence preserves connectedness.

\section{Stable sets and transverse matroids}

\begin{proposition}
\label{allphi}Let $G$ be a looped simple graph, and let $J$ be an independent
set of a transverse matroid of $G$. Then there is a locally equivalent graph
$H$ such that only $\phi_{H}$ elements appear in the image of $J$ under an
induced isomorphism $M[IAS(G)]\rightarrow M[IAS(H)]$.
\end{proposition}

\begin{proof}
According to part 1 of Proposition \ref{iso3} we lose no generality if we
remove all loops in $G$; this avoids unnecessary proliferation of cases.

The proposition is proven by induction of the number $m$ of non-$\phi_{G}$
elements included in $J$. If $m=0$ then $H=G$ satisfies the proposition.

Proceeding inductively, suppose $m>0$ and $v$ is a vertex of $G$ with
$\phi_{G}(v)\notin J$. If $J$ contains $\psi_{G}(v)$, then the image of $J$
under the isomorphism $\beta_{s}^{v}:M[IAS(G)]\rightarrow M[IAS(G_{s}^{v})]$
of Proposition \ref{iso3} contains $\phi_{G_{s}^{v}}(v)$ in addition to every
$\phi_{G_{s}^{v}}(w)$ such that $\phi_{G}(w)\in J$. As $\beta_{s}^{v}(J)$ has
only $m-1$ non-$\phi_{G_{s}^{v}}$ elements, the inductive hypothesis applies.

Suppose instead that every non-$\phi_{G}$ element of $J$ is a $\chi_{G}$
element. We distinguish two cases. Case 1. Suppose $v$ and $w$ are neighbors with
$\chi_{G}(v)$, $\chi_{G}(w)\in J$. Then the image of $J$ under the isomorphism
$\beta_{s}^{v}:M[IAS(G)]\rightarrow M[IAS(G_{s}^{v})]$ contains $\psi
_{G_{s}^{v}}(w)$ in addition to every $\phi_{G_{s}^{v}}(x)$ such that
$\phi_{G}(x)\in J$. Consequently the argument of the preceding paragraph
applies to $\beta_{s}^{v}(J)$. Case 2. Suppose $\chi_{G}(w)\in J$ and there is no
neighbor $v$ of $w$ with $\chi_{G}(v)\in J$. It is impossible that $\phi
_{G}(v)\in J$ $\forall v\in N(w)$, as $\zeta_{w}=\{\chi_{G}(w)\}\cup\{\phi
_{G}(v)\mid v\in N(w)\}$ is a circuit. Consequently there must be a $v\in
N(w)$ with $\phi_{G}(v)\notin J$. Then the image of $J$ under the isomorphism
$\beta_{s}^{v}:M[IAS(G)]\rightarrow M[IAS(G_{s}^{v})]$ contains $\psi
_{G_{s}^{v}}(w)$ in addition to every $\phi_{G_{s}^{v}}(x)$ such that
$\phi_{G}(x)\in J$. Consequently the argument of the preceding paragraph
applies to $\beta_{s}^{v}(J)$.
\end{proof}

A special case of Proposition \ref{allphi} is particularly striking. Let $G$
be a looped simple graph, and let $B$ be a transversal of $W(G)$ that is a basis of $M[IAS(G)]$.
Then Proposition \ref{allphi} tells us that there is a locally equivalent
graph $H$ and an induced isomorphism $\beta:M[IAS(G)]\rightarrow M[IAS(H)]$
such that $\beta(B)=\{\phi_{H}(v)\mid v\in V(H)\}$; we use the notation
$\{\phi_{H}(v)\mid v\in V(H)\}=\Phi(H)$. What makes this special case striking
is the fact that we can use $M[IAS(G)]$ to describe such a graph $H$
explicitly. For each $v\in V(G)$ let $B(v)$, $C(v)$ and $D(v)$ be the elements
of $\tau_{G}(v)$, with $B(v)\in B$ and $C(v),D(v)\notin B$. Then
$\beta(B(v))=\phi_{H}(\beta(v))$. Let $\gamma_{v}$ be the fundamental circuit
of $\beta(C(v))$ with respect to the basis $\Phi(H)$ of $M[IAS(H)]$.
Considering the $\chi_{H}(\beta(v))$ and $\psi_{H}(\beta(v))$ columns of the
matrix $IAS(H)$, it is easy to see that for $w\neq v\in V(G)$, $\beta(v)$ and
$\beta(w)$ are neighbors in $H$ if and only if $\phi_{H}(\beta(w))\in
\gamma_{v}$. (Note that $\gamma_{v}\Delta\{\phi_{H}(\beta(v))\}$ is the
fundamental circuit of $\beta(D(v))$ with respect to $\Phi(H)$, so reversing
the labels of $C(v)$ and $D(v)$ would not affect the validity of the preceding
sentence.) As $\beta$ is a matroid isomorphism, it follows that $v$ and $w$
are neighbors in $H$ if and only if $B(w)$ is an element of the fundamental
circuit of $C(v)$ with respect to $B$ in $M[IAS(G)]$.

We summarize this special case as follows.

\begin{corollary}
\label{allphibasis}Let $G$ be a looped simple graph, and $B\in\mathcal{T}(G)$ a basis of $M[IAS(G)]$. Suppose $C\in\mathcal{T}(G)$ has $B \cap C=\emptyset$. Let $H$ be the graph with $V(H)=V(G)$, in
which two vertices $v$ and $w$ are neighbors if and only if $B(w)$ is an
element of the fundamental circuit of $C(v)$ with respect to $B$ in
$M[IAS(G)]$. Then $G$ and $H$ are locally equivalent, and there is an induced
isomorphism $\beta:M[IAS(G)]\rightarrow M[IAS(H)]$ with $\beta(B)=\Phi(H)$.
\end{corollary}

Notice that we have proven indirectly that $H$ is well defined, i.e., that
$B(w)$ is an element of the fundamental circuit of $C(v)$ with respect to $B$
if and only if $B(v)$ is an element of the fundamental circuit of $C(w)$ with
respect to $B$. A direct proof of this fact would use the same argument as the
proof of Proposition \ref{prop:off_diag_symmetric_repr}.

Corollary \ref{allphibasis} tells us that there is a close connection between
the properties of graphs locally equivalent to $G$ and the properties of
bases of $M[IAS(G)]$ that are transversals of $W(G)$. For instance, $G$ is locally equivalent to a
$d$-regular graph if and only if $M[IAS(G)]$ has a basis $B\in\mathcal{T}(G)$ with
respect to which all fundamental circuits are of size $d+1$ or $d+2$.

The next result implies Theorem \ref{tranmat} of the introduction.

\begin{theorem}
\label{stable}Suppose $G$ is a looped simple graph, $T\in\mathcal{T}(G)$, and
$B$ is a basis of the transverse matroid $M=M[IAS(G)]\mid T$. Let $V_{B}$ be
the subset of $V(G)$ consisting of vertices corresponding to elements of $B$.
Then there is a looped simple graph $H$ that is locally equivalent to $G$,
such that an induced isomorphism $\beta:M[IAS(G)]\rightarrow M[IAS(H)]$ has
these two properties.

\begin{enumerate}
\item The image of $V(G)-V_{B}$ under the associated bijection $\beta
:V(G)\rightarrow V(H)$ is a stable set of $H$.

\item The image of $T$ under $\beta$ is
\[
\{\phi_{H}(\beta(v))\mid v\in V_{B}\}\cup
\bigcup\limits_{v\in V(G)-V_{B}}
\zeta_{H}(\beta(v))\text{ .}
\]
\end{enumerate}
\end{theorem}

\begin{proof}
Proposition \ref{allphi} tells us that there is a graph $H$ locally equivalent
to $G$, such that only $\phi_{H}$ elements appear in the image of $B$ under an
induced isomorphism $\beta:M[IAS(G)]\rightarrow M[IAS(H)]$. According to part
1 of Proposition \ref{iso3}, this property is not affected if we remove all
loops from $H$, so we may just as well assume that $H$ is a simple graph.

We claim that the image of $T-B$ under $\beta$ cannot include any $\phi_{H}$
or $\psi_{H}$ element. Note that the definition of $IAS(H)$ implies that no
set of $\phi_{H}$ columns can sum to 0, as their nonzero entries appear in
different rows.\ Also, no subtransversal consisting of $\phi_{H}$ columns and
a single $\psi_{H}$ column can sum to $0$, because none of the $\phi_{H}$
columns has a nonzero entry in the same row as the diagonal entry of the
$\psi_{H}$ column. It follows that every subtransversal of $W(H)$ containing
some $\phi_{H}$\ elements and a single $\psi_{H}$ element is independent.
Consequently $\beta(T-B)$ cannot contain any $\phi_{H}$ or $\psi_{H}$ element,
because such an element would provide an independent set larger than $B$.

Suppose now that $x,y\in V(G)-V_{B}$ are two vertices of $G$ whose images
under $\beta$ are neighbors in $H$. Then the $\chi_{H}(\beta(x))$ column of
$IAS(H)$ has a nonzero entry in the $\beta(y)$ row. As $y\notin V_{B}$, it
follows that the $\chi_{H}(\beta(x))$ column of $IAS(H)$ is not an element of
the span of the $\phi_{H}(\beta(v))$ columns with $v\in V_{B}$. This is
impossible, since $\beta(B)$ spans $\beta(M)$. Hence there are no such $x$ and
$y$, i.e., $\beta(V(G)-V_{B})$ is a stable set of $H$.
\end{proof}

Theorem \ref{stable} also yields a rather complicated description of the
nullity of an arbitrary subtransversal:

\begin{corollary}
\label{nullity}Let $G$ be a looped simple graph with a subtransversal
$S\in\mathcal{S}(G)$, and let $\nu$ be a non-negative integer. Then $\nu$ is
the nullity of $S$ in $M[IAS(G)]$ if and only if there are a looped simple
graph $H$ and a stable set $X\subseteq V(H)$ that satisfy these three properties:

\begin{enumerate}
\item $H$ is locally equivalent to $G$.
\item $\left\vert X\right\vert =\nu$.
\item An induced isomorphism $\beta:M[IAS(G)]\rightarrow M[IAS(H)]$ has
\[
\bigcup\limits_{x\in X} \zeta_{H}(x) \subseteq \beta(S) \subseteq \{\phi_{H}(v)\mid v\in V(H)-X\} \cup \bigcup\limits_{x\in X} \zeta_{H}(x)\text{ .}
\]
\end{enumerate}
\end{corollary}

\begin{proof}
Suppose the three properties hold, and let $T\in\mathcal{T}(H)$ be the
transversal%
\[
T=\{\phi_{H}(v)\mid v\in V(H)-X\}\cup
\bigcup\limits_{x\in X}
\zeta_{H}(x)\text{ .}%
\]
As $X$ is a stable set in $H$, the transverse matroid $M[IAS(H)]\mid T$ is
represented by a matrix of the form
\[
\bordermatrix{
& X & N(X) & Y \cr
X & 0 & 0 & 0 \cr
N(X) & A & I_1 & 0 \cr
Y & 0 & 0 & I_2 \cr
}.
\]
Here $A$ records adjacencies between vertices in $X$ and vertices in $N(X)$,
and $I_{1},I_{2}$ are identity matrices. The elements of $T$ corresponding to
columns of $A$ and $I_{1}$ are all contained in $\beta(S)$, so $\left\vert
X\right\vert $ is the nullity of $\beta(S)$ in $M[IAS(H)]$.

Suppose conversely that the nullity of $S$ is $\nu$. Let $V_{S}=\{v\in
V(G)\mid S$ contains an element of $\tau_{G}(v)\}$, and let $J$ be an
independent subset of $S$ with $\left\vert S\right\vert -\nu$ elements. If $M$
is a transverse matroid of $G$ that contains $S$, then $M$ has a basis $B$
that contains $J$. The three properties of the statement follow immediately
from Theorem \ref{stable}, with $X=\beta(V_{S}-V_{B})$.
\end{proof}

Notice that in general, choosing a different independent set $J$ will yield a different locally equivalent graph $H$.

\section{Disjoint transversals and bipartite graphs} \label{sec:bipartite}

If $G$ is a looped simple graph we denote by $\Phi(G)$, $X(G)$ and $\Psi(G)$ the transversals of $W(G)$ that include all the $\phi_{G}$, $\chi_{G}$ and $\psi_{G}$ elements (respectively). In this section we  characterize local equivalence to bipartite graphs in several ways, expanding on a result of Bouchet~\cite[Corollary (3.4)]{Bi2}.

\begin{proposition}
\label{basis}Let $T_{1}$ and $T_{2}$ be disjoint transversals of $W(G)$. Then
every independent subtransversal $J\subseteq T_{1}\cup T_{2}$ is contained in
a basis $B$ of $M[IAS(G)]$ that is a transversal contained in $T_{1}\cup T_{2}$.
\end{proposition}

\begin{proof}
Recall that $Q = (M[IAS(G)],W(G))$ is a 3-sheltering matroid. Thus, $\mathcal{Z}(Q)$ is a 3-matroid and so $\mathcal{Z}(Q)[T_{1}\cup T_{2}]$ a 2-matroid and therefore nondegenerate. Recall from Subsection~\ref{ssec:mmatroids} that the bases of nondegenerate multimatroids are transversals.
\end{proof}

\begin{corollary}
Let $T_{1}$ and $T_{2}$ be disjoint transversals of $W(G)$. Then there is a
looped simple graph $H$ locally equivalent to $G$, such that an induced
isomorphism $\beta:M[IAS(G)]\rightarrow M[IAS(H)]$ has $\beta(T_{1}\cup
T_{2})=\Phi(H)\cup X(H)$.
\end{corollary}

\begin{proof}
Let $B$ be a basis of $M[IAS(G)]$ that is a transversal contained in $T_{1}\cup
T_{2}$. Proposition \ref{allphi} tells us that there is a looped simple graph
$H_{0}$ that is locally equivalent to $G$, such that an induced isomorphism
$\beta:M[IAS(G)]\rightarrow M[IAS(H_{0})]$ has $\beta(B)=\Phi(H_{0})$. It
follows that $\beta((T_{1}\cup T_{2})-B)$ is a transversal contained in
$X(H_{0})\cup\Psi(H_{0})$. Loop complementations at vertices of $H_{0}$ that
correspond to elements of $\beta((T_{1}\cup T_{2})-B)\cap\Psi(H_{0})$ will
produce a locally equivalent graph $H$ that satisfies the statement.
\end{proof}

Recall that if $M_1$ and $M_2$ are matroids on disjoint ground sets $W_1$ and $W_2$, then their direct sum $M_{1} \oplus M_{2}$ is the matroid on $W_{1} \cup W_{2}$ whose rank function is given by $r(S)=r_{1}(S\cap W_1)+r_{2}(S\cap W_2)$.

\begin{corollary}
\label{bip}
Let $G$ be a looped simple graph. Then any one of the following conditions is
equivalent to the others:

\begin{enumerate}
\item $G$ is locally equivalent to a bipartite graph.

\item $G$ has a pair of disjoint transversals with $r(T_{1})+r(T_{2}%
)=\left\vert V(G)\right\vert $.

\item $G$ has a pair of disjoint transversals with%
\[
M[IAS(G)]\mid(T_{1}\cup T_{2})=(M[IAS(G)]\mid T_{1})\oplus(M[IAS(G)]\mid
T_{2})\text{.}
\]

\item $G$ has a pair of disjoint transversals with%
\[
(M[IAS(G)]\mid T_{1})\cong(M[IAS(G)]\mid
T_{2})^{\ast}\text{.}
\]

\end{enumerate}
\end{corollary}

\begin{proof}
If $G$ is a bipartite graph with vertex-classes $V_{1}$ and $V_{2}$, let
$T_{1}$ and $T_{2}$ be the transversals of $W(G)$ given by
\[
T_{i}=\{\phi_{G}(v)\mid v\in V_{i}\}\cup
\bigcup\limits_{v\notin V_{i}}
\zeta_{G}(v)\text{ .}
\]
Let $M_1=M[IAS(G)]\mid T_{1}$, $M_2=M[IAS(G)]\mid T_{2}$ and $M_{12}=M[IAS(G)]\mid(T_{1}\cup T_{2})$. Then there is a matrix $A$ such that $M_1$ and $M_2$ are represented (respectively) by 
\[
\bordermatrix{
& V_1 & V_2 \cr
V_1 & I_1 & A \cr
V_2 & 0 & 0 \cr
}
\qquad \text{and} \qquad 
\bordermatrix{
& V_1 & V_2 \cr
V_1 & 0 & 0 \cr
V_2 & A^{\text{T}} & I_2 \cr
}
\text{,}
\]
where $I_{1}$ and $I_{2}$ are identity matrices. It follows that $M_1 \cong (M_{2})^{\ast}$~\cite[Theorem 2.2.8]{O}. Also,
$M_{12}$ is represented by 
\[
\bordermatrix{
& V_1 & V_2 & V_1 & V_2 \cr
V_1 & I_1 & A & 0 & 0 \cr
V_2 & 0 & 0 & A^{\text{T}} & I_2 \cr
}\text{.}
\]
As no row of this matrix has a nonzero entry in a column corresponding to an element of $T_{1}$ and also a nonzero entry in a column corresponding to an element of $T_{2}$, $M_{12}$ is the direct sum of $M_{1}$ and $M_{2}$. 

To verify the implications $1\Rightarrow3$ and $1\Rightarrow4$, note that if $G$ is locally
equivalent to a bipartite graph $H$ then as was just observed, $H$ has a pair
of transversals that satisfy conditions 3 and 4. The images of these
transversals under an induced isomorphism $M[IAS(H)] \to M[IAS(G)]$ are transversals of $G$ that satisfy conditions 3 and 4.

The implication $4\Rightarrow2$ is obvious. To verify the implication $3\Rightarrow2$, note that if $T_{1}$ and $T_{2}$
satisfy condition 3 then $r(T_{1})+r(T_{2})$ is the rank of $M[IAS(G)]\mid
(T_{1}\cup T_{2})$, which is $\left\vert V(G)\right\vert $ by Proposition~\ref{basis}.

It remains to verify the implication $2\Rightarrow1$. Suppose that $G$ has a pair of disjoint transversals with $r(T_{1})+r(T_{2})=\left\vert
V(G)\right\vert $. By Proposition \ref{basis}, $M[IAS(G)]$ has a transverse
basis $B\subseteq T_{1}\cup T_{2}$. Then $B_{1}=B\cap T_{1}$ and $B_{2}=B\cap T_{2}$ are both independent sets of $M[IAS(G)]$; as their cardinalities sum to
$r(T_{1})+r(T_{2})$, each $B_{i}$ must be a maximal independent subset of
$T_{i}$. By Proposition \ref{allphi}, there is a graph $H$ that is locally
equivalent to $G$, such that an induced isomorphism $\beta
:M[IAS(G)]\rightarrow M[IAS(H)]$ has $\beta(B)=\Phi(H)$. For $i\in\{1,2\}$ let
$V_{i}=\{v\in V(H)\mid\phi_{H}(v)\in\beta(B_{i})\}$.

As $\beta(B_{1})\subseteq\Phi(H)$, no column of $IAS(H)$ with a nonzero entry
in a row corresponding to a vertex outside $V_{1}$ is in the span of the
columns corresponding to elements of $\beta(B_{1})$. As $r(\beta
(T_{1}))=\left\vert B_{1}\right\vert $, every column corresponding to an
element of $\beta(T_{1}-B_{1})$ is in the span of the columns corresponding to
elements of $\beta(B_{1})$; consequently no element of $\beta(T_{1}-B_{1}) $
corresponds to a column that includes a nonzero entry in a row corresponding
to an element of $V_{2}$, so no two elements of $V_{2} $ are neighbors in $H$.
The same argument applies if we reverse the roles of $B_{1} $ and $B_{2}$, so
$H$ is a bipartite graph.
\end{proof}

For instance, the graph of Figure \ref{circh1f2} might at first glance seem to
resemble the wheel graph $W_{5}$. But in fact, it is quite different. A
computer search indicates that the smallest rank of a transversal of $W_{5}$
is 4, but the pictured graph has two disjoint transversals of rank 3. We leave finding them as an exercise for the reader. Here's a hint: local
complementations at the degree-2 vertices produce a bipartite graph.%

\begin{figure}[ptb]
\centering
\includegraphics[
trim=4.013889in 8.303725in 2.810742in 1.210616in,
natheight=11.005600in,
natwidth=8.496800in,
height=1.1459in,
width=1.2825in
]%
{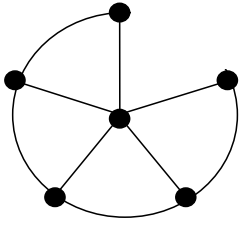}
%
\caption{A graph locally equivalent to a bipartite graph.}
\label{circh1f2}
\end{figure}

Corollary~\ref{bip} has an interesting consequence, having as a special case a result regarding bicycle spaces of planar graphs~\cite[Theorem 17.3.5]{GR}. The connection with planar graphs arises from the fact that medial graphs of planar graphs are associated with bipartite circle graphs; see the sequel to the present paper~\cite{BT2} for details. 

\begin{corollary}
\label{bic}
Suppose $T_1$ and $T_2$ are disjoint transversals of $G$, which satisfy Corollary~\ref{bip}. Let $T_3=W(G)\setminus (T_{1} \cup T_{2})$, and for $1 \leq i \leq 3$ let $M_i$ be the matroid on $V(G)$ defined by $M[IAS(G)]\mid T_{i}$, using the obvious bijection between $T_i$ and $V(G)$. Then $M_1$ and $M_2$ have the same bicycle space, which equals the cycle space of $M_3$.
\end{corollary}

\begin{proof}
$M_1$, $M_2$ and $M_3$ are represented by three matrices
\[
A_1=\bordermatrix{
& V_1 & V_2 \cr
V_1 & I_1 & A \cr
V_2 & 0 & 0 \cr
}
\text{,} \qquad 
A_2=\bordermatrix{
& V_1 & V_2 \cr
V_1 & 0 & 0 \cr
V_2 & A^{\text{T}} & I_2 \cr
}
\qquad \text{and} \qquad 
A_3=\bordermatrix{
& V_1 & V_2 \cr
V_1 & I_1 & A \cr
V_2 & A^{\text{T}} & I_2 \cr
}
\text{,}
\]
respectively. For $1 \leq i \leq 3$ let $Z_i$ be the cycle space of $M_i$, i.e., the orthogonal complement of the row space of $A_i$. Clearly then $Z_3$=$Z_1 \cap Z_2$. Moreover, $Z_1$ and $Z_2$ are orthogonal complements of each other~\cite[Proposition 2.2.23]{O}, so $Z_1 \cap Z_2$ is the bicycle space of both $M_1$ and $M_2$.
\end{proof}

While transverse matroids of isotropic matroids are (of course) binary, Corollary~\ref{bic} extends to quaternary matroids by generalizing the notion of an isotropic matroid in a suitable way from $GF(2)$ to $GF(4)$; details are provided in~\cite[Section~3]{B} (see also \cite{BH4}, formulated there in terms of delta-matroids). 

It is also worth mentioning that the converse of Corollary~\ref{bic} does not hold. That is, the condition ``$G$ has pairwise disjoint transversals $T_1,T_2,T_3$ such that $M_1$ and $M_2$ have the same bicycle space, which equals the cycle space of $M_3$'' is not sufficient to guarantee that $G$ satisfies Corollary~\ref{bip}. For instance, $M[IAS(C_5)]$ has many sets of three pairwise disjoint transversal bases; in the notation of Section 3, one such triple includes $T_1=\{\phi_1,\phi_2,\phi_3,\phi_4,\psi_5\}$, $T_2=\{\chi_1,\chi_2,\psi_3,\chi_4,\chi_5\}$ and $T_3=\{\psi_1,\psi_2,\chi_3,\psi_4,\phi_5\}$. Any such transversal bases satisfy the condition quoted above, because the cycle and bicycle spaces of $M_1,M_2$ and $M_3$ are all $\{0\}$. But inspecting the matrix $IAS(C_5)$ displayed in Section 3, we see that no two columns are the same; as $GF(2)^2$ has only four elements, it follows that there is no transversal of rank $\leq 2$.  Consequently $C_5$ does not satisfy Corollary~\ref{bip}.

\section{Neighborhood circuits and transverse circuits}

Theorem \ref{trancirc} of the introduction follows immediately from Corollary \ref{nullity}, with $\nu=1$. Corollary \ref{nullity} is also useful when $\nu>1$. For instance, the following four results indicate that we can use transverse circuits to detect certain types of vertex pairs in locally equivalent graphs.

\begin{corollary}
\label{circuits1}Suppose $G$ is a looped simple graph, and $k_{1},k_{2}%
\in\mathbb{N}$. Then these statements are equivalent.

\begin{enumerate}
\item $G$ is locally equivalent to some graph $H$ with nonadjacent vertices of
degrees $k_{1}-1$ and $k_{2}-1$, which do not share any neighbor.

\item $G$ has a transverse matroid with two disjoint circuits of sizes $k_{1}
$ and $k_{2}$, whose union contains no other circuit.
\end{enumerate}
\end{corollary}

\begin{proof}
Suppose $G$ satisfies condition 1, and let $v\,$\ and $w$ be vertices of $H$
as described. Then
\[
\zeta_{H}(v)\cup\zeta_{H}(w)\cup\{\phi_{H}(x)\mid x\in V(H)-\{v,w\}\}
\]
is a transversal of $W(H)$ that contains only two circuits, $\zeta_{H}(v)$ and
$\zeta_{H}(w)$. The inverse image of this transversal under an induced
isomorphism $\beta:M[IAS(G)]\rightarrow M[IAS(H)]$ satisfies condition 2.

For the converse, let $S$ be the union of the two circuits mentioned in
condition 2. Then $S$ is a subtransversal whose nullity is 2. Corollary
\ref{nullity} tells us that there is a graph $H$ that is locally equivalent to
$G$, such that the images of the two circuits mentioned in condition 2 under
an induced isomorphism $M[IAS(G)]\rightarrow M[IAS(H)]$ are both neighborhood circuits.
\end{proof}

\begin{corollary}
\label{diam}These two statements about a looped simple graph $G$ are equivalent.

\begin{enumerate}
\item $G$ is locally equivalent to a graph of diameter $>2$.

\item $G$ has a transverse matroid with two disjoint circuits, whose union
contains no other circuit.
\end{enumerate}
\end{corollary}

\begin{proof}
This result follows immediately from Corollary \ref{circuits1}, as a graph has
diameter $>2$ if and only if it has a pair of nonadjacent vertices which do
not share any neighbor.
\end{proof}

\begin{corollary}
\label{circuits2}Suppose $G$ is a looped simple graph, and $k_{1},k_{2}%
\in\mathbb{N}$. Then these statements are equivalent.

\begin{enumerate}
\item $G$ is locally equivalent to some graph $H$ with nonadjacent vertices of
degrees $k_{1}-1$ and $k_{2}-1$, which share a neighbor.

\item $G$ has a transverse matroid of nullity 2, with distinct, intersecting
circuits of sizes $k_{1}$ and $k_{2}$.
\end{enumerate}
\end{corollary}

\begin{proof}
Let $v\,$\ and $w$ be vertices of a graph $H$ that is locally equivalent to
$G$, as described in condition 1. Then the inverse image of
\[
\zeta_{H}(v)\cup\zeta_{H}(w)\cup\{\phi_{H}(x)\mid x\in V(H)-\{v,w\}\}
\]
under an induced isomorphism $\beta:M[IAS(G)]\rightarrow M[IAS(H)]$ is a
transverse matroid of $G$, which satisfies condition 2.

For the converse, let $M$ be a transverse matroid of $G$ of nullity 2, and
suppose $\gamma_{1}$ and $\gamma_{2}$ are distinct, intersecting circuits of
$M$ with $\left\vert \gamma_{1}\right\vert =k_{1}$ and $\left\vert \gamma
_{2}\right\vert =k_{2}$. The columns of $IAS(G)$ corresponding to elements of
$\gamma_{1}$ sum to 0, and so do the columns corresponding to elements of
$\gamma_{2}$. Consequently the columns of $IAS(G)$ corresponding to elements
of $\gamma_{1}\Delta\gamma_{2}$ also sum to 0. If $M$ were to have a circuit
$\gamma_{3}\subsetneqq\gamma_{1}\Delta\gamma_{2}$, then it would also have a
circuit $\gamma_{4}\subseteq(\gamma_{1}\Delta\gamma_{2})-\gamma_{3}$, because
the columns of $IAS(G)$ corresponding to elements of $(\gamma_{1}\Delta
\gamma_{2})-\gamma_{3}$ would sum to 0. Then an independent set $J$ of $M$
would have to exclude an element $x$ of $\gamma_{3}$ and an element $y$ of
$\gamma_{4}$, and at least one more element: if $x,y\in$ $\gamma_{1}%
-\gamma_{2}$ then $J$ would have to exclude an element $z$ of $\gamma_{2}$, if
$x,y\in\gamma_{2}-\gamma_{1}$ then $J$ would have to exclude an element $z$ of
$\gamma_{1}$, and if $x\in$ $\gamma_{1}-\gamma_{2}$ and $y\in$ $\gamma
_{2}-\gamma_{1}$ then $J$ would have to exclude some element $z$ of
$\gamma_{1}\cup\gamma_{3}-\{x\}$, as the circuit elimination property
guarantees that $\gamma_{1}\cup\gamma_{3}-\{x\}$ is dependent. As the nullity
of $M$ is only 2, we conclude by contradiction that $\gamma_{1}\Delta
\gamma_{2}$ is a circuit of $M$.

Let $J$ be a subset of $M$ obtained by removing one element of $\gamma_{1}-\gamma_{2}$ and also removing one element of $\gamma_{2}-\gamma_{1}$. Then $J $ is an independent set of $M[IAS(G)]$. Applying the last paragraph of the proof of Corollary \ref{nullity} to $J$, we conclude that there is a graph $H$ locally equivalent to $G$, such that the images of $\gamma_{1}$ and $\gamma_{2}$ under an induced isomorphism $M[IAS(G)]\rightarrow M[IAS(H)]$ are both neighborhood circuits.
\end{proof}

\begin{corollary}
\label{circuits3}Let $G$ be a looped simple graph, and let $k_{1},k_{2}%
\in\mathbb{N}$. Then these statements are equivalent.

\begin{enumerate}
\item $G$ is locally equivalent to a graph with adjacent vertices of degrees
$k_{1}-1$ and $k_{2}-1$.

\item $M[IAS(G)]$ has two transverse circuits $\gamma_{1}$ and $\gamma_{2}$
such that $\left\vert \gamma_{1}\right\vert =k_{1}$, $\left\vert \gamma
_{2}\right\vert =k_{2}$, the largest subtransversals contained in $\gamma
_{1}\cup\gamma_{2}$ are of size $\left\vert \gamma_{1}\cup\gamma
_{2}\right\vert -2$, and two of these largest subtransversals are independent
sets of $M[IAS(G)]$.
\end{enumerate}
\end{corollary}

\begin{proof}
Suppose $G$ is locally equivalent to a graph $H$ with adjacent vertices
$v_{1}$ and $v_{2}$, of degrees $k_{1}-1$ and $k_{2}-1$. Then the neighborhood
circuits $\zeta_{H}(v_{1})$ and $\zeta_{H}(v_{2})$ are transverse circuits of
$H$ such that $\left\vert \zeta_{H}(v_{1})\right\vert =k_{1}$ and $\left\vert
\zeta_{H}(v_{2})\right\vert =k_{2}$. As $\phi_{H}(v_{1})\in\zeta_{H}(v_{2})$
and $\phi_{H}(v_{2})\in\zeta_{H}(v_{1})$, the largest subtransversals
contained in $\zeta_{H}(v_{1})\cup\zeta_{H}(v_{2})$ are of size $\left\vert
\zeta_{H}(v_{1})\cup\zeta_{H}(v_{2})\right\vert -2$. One independent
subtransversal of maximum size contains only $\phi_{H}$ elements, and the
other includes one of $\chi_{H}(v_{1}),\psi_{H}(v_{1})$ and one of $\chi
_{H}(v_{2}),\psi_{H}(v_{2})$, along with every $\phi_{H}(w)$ such that
$w\in(N_{H}(v_{1})\cup N_{H}(v_{2}))-\{v_{1},v_{2}\}$. The inverse images of
$\zeta_{H}(v_{1})$ and $\zeta_{H}(v_{2})$ under an induced isomorphism
$M[IAS(G)]\rightarrow M[IAS(H)]$ are transverse circuits of $G$ that satisfy
the requirements of the statement.

Conversely, suppose $G$ has transverse circuits $\gamma_{1}$ and $\gamma_{2}$
as in the statement, and let $S\subseteq\gamma_{1}\cup\gamma_{2}$ be an
independent subtransversal of size $\left\vert \gamma_{1}\cup\gamma
_{2}\right\vert -2$. As $S$ is independent, it does not contain any circuit;
hence $S$ must exclude at least one element of $\gamma_{1}$ and at least one
element of $\gamma_{2}$. Proposition \ref{allphi} tells us that there is a
locally equivalent graph $H$ such that the image of $S$ under an induced
isomorphism $M[IAS(G)]\rightarrow M[IAS(H)]$ contains only $\phi_{H}$
elements. The images of the two elements of $\gamma_{1}\cup\gamma_{2}-S$ must
correspond to columns of $IAS(H)$ with diagonal entries equal to 0, as the
images of $\gamma_{1}$ and $\gamma_{2}$ are dependent. It follows that the
images of $\gamma_{1}$ and $\gamma_{2}$ are neighborhood circuits of vertices
$v_{1}$ and $v_{2}$ of degrees $\left\vert \gamma_{1}\right\vert -1$ and
$\left\vert \gamma_{2}\right\vert -1$, respectively. The second independent
subtransversal of size $\left\vert S\right\vert $ must exclude both $\phi
_{H}(v_{1})$ and $\phi_{H}(v_{2})$, for if it were to contain either of them it would
contain $\zeta_{H}(v_{1})$ or $\zeta_{H}(v_{2})$, and consequently it
would be dependent. This second subtransversal would not be independent if
$v_{1}$ and $v_{2}$ were not adjacent.
\end{proof}

\section{An example}
\label{wheel}

Corollaries~\ref{nullnu} and~\ref{circk} can be used to provide particularly simple descriptions of some local equivalence classes. For instance, a search using the matroid module for Sage \cite{sageMatroid, sage} indicates that if a graph $G$ of order $\leq 6$  has no transverse circuit of size $<4$, then $G$ is locally equivalent to the wheel graph $W_{5}$. The local equivalence class of $W_5$ is also characterized by the relatively small nullities of its transverse matroids (the largest nullity is 2). These observations yield several characterizations of this local equivalence class:

\begin{proposition}
\label{w5}Let $G$ be a looped simple graph with $n\leq6$ vertices. Then any
one of the following properties implies the others.

\begin{enumerate}
\item $G$ is locally equivalent to the wheel graph $W_{5}$.

\item $G$ is not locally equivalent to any graph with a vertex of degree
$\leq2$.

\item $G$ is not locally equivalent to any graph with a stable set of size
$\geq n-3$.

\item $G$ has no transverse circuit of size $\leq3$.

\item $G$ has no transverse matroid of rank $\leq3$.
\end{enumerate}
\end{proposition}

The local equivalence class of $W_{5}$ is important in Bouchet's famous characterization of circle graphs by obstructions \cite{Bco}. In sequels to the present paper \cite{BT2, BT3} we extend Proposition \ref{w5} and provide several new characterizations of circle graphs.

\begin{figure}[ptb]
\centering
\includegraphics[
trim=3.880489in 8.297122in 2.810742in 1.208415in,
natheight=11.005600in,
natwidth=8.496800in,
height=1.1519in,
width=1.382in
]
{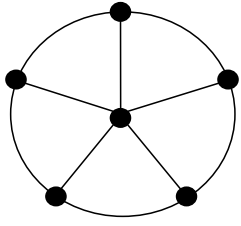}
\caption{The wheel graph $W_{5}$.}
\label{circh1f1}
\end{figure}

\section{Matroid minors and vertex-minors} \label{sec:minors}

Isotropic matroids of graphs constitute a very limited class of binary
matroids. The limitation is clear even if we note only that they are
$3n$-element matroids, as this implies that when a single element is
contracted or deleted from an isotropic matroid, the result cannot be an
isotropic matroid.

There is a special minor operation that is appropriate for isotropic matroids,
which involves removing entire vertex triples.

\begin{definition}
\label{isominor}Let $G$ be a looped simple graph, let $S$ be a subtransversal
of $W(G)$, and let $S^{\prime}$ contain the other $2\left\vert S\right\vert $
elements of $W(G)$ that\ correspond to the same vertices of $G$ as elements of
$S$. Then
\[
(M[IAS(G)]/S)-S^{\prime}%
\]
is the \emph{isotropic minor }of $G$ obtained by contracting $S$ and deleting
$S^{\prime}$.
\end{definition}

Notice that if $S$ is specified then it is not necessary to explicitly mention
$S^{\prime}$, as $S^{\prime}$ is determined by $S$. Consequently we may
sometimes refer simply to the isotropic minor obtained by contracting $S$. By
the way, the definition is consistent with Bouchet's definitions of minors of
isotropic systems \cite{Bi1} and multimatroids \cite{B2}.

\begin{definition}
A \emph{vertex-minor} of a looped simple graph $G$ is a graph obtained from
$G$ through some sequence of local complementations, loop complementations
and vertex deletions.
\end{definition}

\begin{theorem}
\label{isominorspec}(\cite[Section 7.1]{Tnewnew}) The isotropic minors of $G$
are precisely the isotropic matroids of vertex-minors of $G$.
\end{theorem}

In particular, if $v\in V(G)$ is unlooped and $w\in N_{G}(v)$ then the
isotropic minor of $G$ obtained by contracting $\chi_{G}(v)$ is isomorphic to
$M[IAS(((G_{s}^{v})_{s}^{w})_{s}^{v}-v)]$, the isotropic minor of $G$ obtained
by contracting $\psi_{G}(v)$ is $M[IAS(G_{ns}^{v}-v)]$, and the isotropic
minor of $G$ obtained by contracting $\phi_{G}(v)$ is $M[IAS(G-v)]$. Notice
that we only say \textquotedblleft is isomorphic to\textquotedblright\ in the
first case, because in that case the matroid isomorphism requires a
permutation of the $\phi,\chi,\psi$ labels in the vertex triple $\tau_{G}(w)$.
No such label change is needed in the other two cases. We refer to
\cite{Tnewnew} for details.

In contrast, it turns out that all minors of transverse matroids are
transverse minors, in an appropriate sense.

\begin{proposition}
\label{minors}Let $G$ be a looped simple graph with a transverse matroid $M$.
Then every matroid minor of $M$ is a transverse matroid of some vertex-minor
of $G$.
\end{proposition}

\begin{proof}
It suffices to verify this for the minors obtained by contracting and deleting
a single element $m$ of $M$. For $M/m$, the result is obvious: $M/m$ is a
transverse matroid of the isotropic minor of $M[IAS(G)]$ obtained by
contracting $m$ and deleting the other two elements of the corresponding
vertex triple.

To realize $M-m$ as a transverse matroid of an isotropic minor of $G$, we
recall the triangle property of isotropic matroids: if $m^{\prime}$ and
$m^{\prime\prime}$ are the other two elements of the vertex triple that
contains $m$, then one of the three transverse matroids $M$, $(M-m)\cup
\{m^{\prime}\}$, $(M-m)\cup\{m^{\prime\prime}\}$ has the same rank as $M-m$,
and the other two are of rank $r(M-m)+1$. In any case we may presume that
$r((M-m)\cup\{m^{\prime}\})=r(M-m)+1$. Then $m^{\prime}$ is a coloop of
$(M-m)\cup\{m^{\prime}\}$, so $M-m=((M-m)\cup\{m^{\prime}\})-m^{\prime}$ is
isomorphic to $((M-m)\cup\{m^{\prime}\})/m^{\prime}$. As observed in the
preceding paragraph, the latter matroid is a transverse matroid of an
isotropic minor of $G$.
\end{proof}

\begin{corollary}
\label{minorclass}Let $\mathcal{M}$ be a class of binary matroids that is
closed under matroid minors, and let $\mathcal{G}_{\mathcal{M}}$ be the family
of looped simple graphs whose transverse matroids are all from $\mathcal{M}$.
Then $\mathcal{G}_{\mathcal{M}}$ is closed under vertex-minors.
\end{corollary}

It is regrettable that the most important vertex-minor-closed family of looped
simple graphs, the looped circle graphs, cannot be described in this easy way.
(Looped circle graphs constitute a proper subfamily of $\mathcal{G}%
_{cographic}$.) We discuss this important family in sequels to the present
paper \cite{BT2, BT3}.

\section{Parallel reductions and distance hereditary graphs} \label{sec:par_red_hereditary}

Recall some elementary definitions of matroid theory. A \emph{loop} of a matroid is an element that is excluded from every basis. Two non-loop elements $x$ and $y$ are \emph{parallel} if $\{x,y\}$ is a circuit; equivalently, no basis includes them both. We also consider all loops to be parallel to each other. Dually, a \emph{coloop} is an element that is included in every basis, and we consider all coloops to be \emph{in series} with each other. Two non-coloop elements are in series if no basis excludes them both.

It is a simple matter to recognize loops and parallels in matroids represented by binary matrices: a column represents a loop if all of its entries are 0, and two columns represent parallels if all of their entries are the same. In general it is not quite so easy to recognize coloops and elements in series, but this will not concern us because isotropic matroids have no coloops, and contain no series pairs that are not also parallel:

\begin{proposition}
\label{serieschar}
Let $G$ be a looped simple graph. 
\begin{itemize}
\item No element of $M[IAS(G)]$ is a coloop.

\item Two elements of $M[IAS(G)]$ are in series if and only if they are the parallel, non-loop elements of the vertex triple of an isolated vertex of $G$.
\end{itemize}
\end{proposition}

\begin{proof}
Suppose first that $\rho$ is a coloop of $M[IAS(G)]$. As $\Phi
(G)=\{\phi_{G}(v)\mid v\in V(G)\}$ is a basis of $M[IAS(G)]$, $\rho = \phi_{G}(v)$ for some $v \in V(G)$. Let $\rho^{\prime}$ be the one of $\chi_{G}(v),\psi_{G}(v)$ which corresponds to a column of $IAS(G)$ with a nonzero $v$ entry. Then the symmetric difference $\Phi
(G) \Delta \{\rho,\rho^{\prime}\}$ is a basis of $M[IAS(G)]$, and it does not contain $\rho$. We conclude by contradiction that $\rho$ is not a coloop.

If $v$ is an isolated vertex of $G$ then the columns of $IAS(G)$ representing the two non-loop elements of the vertex triple $\tau_{G}(v)$ are the same, and they
are the only columns of $IAS(G)$ with nonzero entries in the $v$ row. Consequently the
corresponding elements of $M[IAS(G)]$ are parallel, and they are in series.

Now, suppose $\rho$ and $\sigma$ are in series in $M[IAS(G)]$. The basis $\Phi(G)$ must include at least one of $\rho,\sigma$; say $\rho
=\phi_{G}(v)$. We claim that $\sigma \notin \Phi(G)$. Suppose the claim is incorrect, and $\sigma = \phi_{G}(w)$. If $w$ is a neighbor of $v$, choose $\rho^{\prime} \neq \rho \in \tau_{G}(v)$ and $\sigma^{\prime} \neq \sigma \in \tau_{G}(w)$ so that the corresponding columns of $IAS(G)$ have 0 entries in the $v$ and $w$ rows (respectively). Then $B = \Phi(G) \Delta \{\rho,\rho^{\prime},\sigma,\sigma^{\prime}\}$ is a basis of $M[IAS(G)]$, because the columns of $IAS(G)$ corresponding to $\rho^{\prime}$ and $\sigma^{\prime}$ have nonzero entries in the $w$ and $v$ rows (respectively). But $B$ contains neither $\rho$ nor $\sigma$, an impossibility. If $w$ is not a neighbor of $v$, instead, then choose $\rho^{\prime} \neq \rho \in \tau_{G}(v)$ and $\sigma^{\prime} \neq \sigma \in \tau_{G}(w)$ so that the corresponding columns of $IAS(G)$ have nonzero entries in the $v$ and $w$ rows (respectively). Then again, $\Phi(G) \Delta \{\rho,\rho^{\prime},\sigma,\sigma^{\prime}\}$ is a basis of $M[IAS(G)]$ which contains neither $\rho$ nor $\sigma$; and again, this is impossible. We conclude by contradiction that the claim $\sigma \notin \Phi(G)$ must be correct.

If $v$ is not isolated in $G$, let $x$ be a
neighbor of $v$. Then the columns of $IAS(G)$ corresponding to $\chi_{G}(x)$
and $\psi_{G}(x)$ both have nonzero entries in the $v$ row. Choose one of
$\chi_{G}(x)$, $\psi_{G}(x)$ that is not equal to $\sigma$, and denote it
$\rho^{\prime}$. Then $\Phi(G)\Delta\{\rho,\rho^{\prime}\}$ is a basis of
$M[IAS(G)]$ which contains neither $\rho$ nor $\sigma$, an impossibility. We conclude that $v$ must be isolated.

Suppose $\sigma$ is an element of a vertex triple $\tau_{G}(w)$, where
$w\neq v$. Let $\rho^{\prime}$ be the non-loop element of $\tau_{G}(v)$, other
than $\rho$. As $\sigma\neq\phi_{G}(w)$, $\Phi(G)\Delta\{\rho,\rho^{\prime
}\}\,$ is a basis of $M[IAS(G)]$ which excludes both $\rho$ and $\sigma$, an impossibility.

The only remaining possibility is that $\sigma$ is the non-loop element of $\tau_{G}(v)$ other than $\rho$. As noted in the second paragraph of the proof, it follows that $\rho$ and $\sigma$ are both parallel and in series.
\end{proof}

In contrast, there are several kinds of parallels in isotropic matroids. 

\begin{proposition}
\label{loopchar}
\label{loops} Let $G$ be a looped simple graph. An element of $M[IAS(G)]$ is a loop if it is the $\chi_{G}$ element of an isolated unlooped vertex, or the $\psi_{G}$ element of an isolated looped vertex.
\end{proposition}

\begin{proof}
An element of $M[IAS(G)]$ is a loop if and only if every entry of the corresponding column of $IAS(G)$ is $0$.
\end{proof}

\begin{proposition}
\label{parallels} Let $G$ be a looped simple graph. Two non-loop elements of $M[IAS(G)]$ are parallel if and only if they fall into one of these four categories:

\begin{enumerate}
\item If $v\in V(G)$ is isolated then the two non-loop elements of the vertex triple $\tau_{G}(v)$ are parallel.

\item If $v\neq w\in V(G)$ and $N_{G}(v)=N_{G}(w)\neq\emptyset$ then the
vertex triples $\tau_{G}(v)$ and $\tau_{G}(w)$ contain a parallel pair, which
includes the two elements whose corresponding columns have $0$ in both the $v$
row and the $w$ row.

\item If $v\neq w\in V(G)$ and $N_{G}(v)\cup\{v\}=N_{G}(w)\cup\{w\}$ then
$\tau_{G}(v)$ and $\tau_{G}(w)$ contain a parallel pair, which includes the
two elements whose corresponding columns have $1$ in both the $v$ row and the
$w$ row.

\item If $v\neq w\in V(G)$ and $N_{G}(v)=\{w\}$ then $\tau_{G}(v)$ and
$\tau_{G}(w)$ contain a parallel pair, which includes $\phi_{G}(w)$ and the
element of $\tau_{G}(v)$ whose corresponding column has $0$ in the $v$ row.
\end{enumerate}
\end{proposition}

\begin{proof}
Two non-loop elements of $M[IAS(G)]$ are parallel if and only if the
corresponding columns of $IAS(G)$ have the same entries, at least one of which
is not 0.
\end{proof}

Notice that in case 1, the two non-loop elements of $\tau_{G}(v)$ correspond to the only two columns of $IAS(G)$ with nonzero entries in the $v$ row. Consequently these two elements constitute a component of $M[IAS(G)]$.

Vertex pairs of the types mentioned in cases 2 and 3 are \emph{nonadjacent
twins} and \emph{adjacent twins}, respectively. In case 4, $v$ is
\emph{pendant} on $w$. Notice that if $v$ and $w$ fall under case 2 or case 3
in $G$, then they fall under under case 4 in a graph locally equivalent to
$G$. For if $v$ and $w$ are adjacent twins in $G$, then $v$ is pendant on $w$
in $G_{s}^{w}$ and $G_{ns}^{w}$; and if $v$ and $w$ are nonadjacent twins with
a common neighbor $x$ in $G$, then $v$ and $w$ are adjacent twins in
$G_{s}^{x}$ and $G_{ns}^{x}$. Also, if $N_{G}(v)\cup\{v\}=N_{G}(w)\cup
\{w\}=\{v,w\}$ then cases 3 and 4 both apply.

\begin{corollary}
\label{parallelcor}If $\rho$ and $\sigma$ are parallel non-loop elements of
$M[IAS(G)]$ then one of these cases holds.

\begin{enumerate}
\item A single vertex triple contains both $\rho$ and $\sigma$. Moreover, the
submatroid $M[IAS(G)]\mid\{\rho,\sigma\}$ is a component of $M[IAS(G)]$.

\item Two distinct vertex triples $\{\rho,\rho^{\prime},\rho^{\prime\prime}\}$
and $\{\sigma,\sigma^{\prime},\sigma^{\prime\prime}\}$ contain $\rho$ and
$\sigma$. Moreover, (a) there is a compatible automorphism of $M[IAS(G)]$ that
interchanges $\rho$ and $\sigma$, interchanges $\{\rho,\rho^{\prime}%
,\rho^{\prime\prime}\}$ and $\{\sigma,\sigma^{\prime},\sigma^{\prime\prime}%
\}$, and preserves all other vertex triples and (b) there is a compatible
automorphism of $M[IAS(G)]$ that preserves all vertex triples, fixes $\rho$
and $\sigma$, interchanges $\rho^{\prime}$ and $\rho^{\prime\prime}$, and
interchanges $\sigma^{\prime}$ and $\sigma^{\prime\prime}$.
\end{enumerate}
\end{corollary}

\begin{proof}
If case 1 of Proposition \ref{parallels} holds, then case 1 of this statement holds.

Suppose case 4 of Proposition \ref{parallels} holds, i.e., $v$ is pendant on
$w$ in $G$. For notational convenience, suppose that neither $v$ nor $w$ is
looped; then $\{\rho,\sigma\}=\{\chi_{G}(v),\phi_{G}(w)\}$. As the columns of
$IAS(G)$ corresponding to these elements are identical, the transposition
$(\chi_{G}(v)\phi_{G}(w))$ defines an automorphism of $M[IAS(G)]$.

Note that the only four columns of $M[IAS(G)]$ with nonzero entries in the $v$
row are the columns corresponding to elements of the set $S=\{\phi_{G}(v)$,
$\psi_{G}(v)$, $\chi_{G}(w)$, $\psi_{G}(w)\}$. Consequently, every element of the cycle space of
$M[IAS(G)]$ includes an even number of elements of $S$. Note also that $S$ is
an element of the cycle space, i.e., the sum of these four columns is $0$. Consequently, if an element of the cycle space
of $M[IAS(G)]$ contains precisely two elements of $S$, then we obtain a new element of the cycle space by replacing these two elements with the other two elements of $S$. It
follows that if a permutation $\pi$ of $S$ is the composition of two disjoint
transpositions, then $\pi$ defines an automorphism of $M[IAS(G)]$.

Consequently the permutation $(\phi_{G}(v)\psi_{G}(v))(\chi_{G}(w)\psi
_{G}(w))$ is an example of a compatible automorphism of $M[IAS(G)]$ that
satisfies part 2(b) of the statement, and
\[
(\chi_{G}(v)\phi_{G}(w))(\phi_{G}(v)\chi_{G}(w))(\psi_{G}(v)\psi_{G}(w))
\]
is an example of an automorphism that satisfies part 2(a).

If case 2 or case 3 of Proposition \ref{parallels} holds then as noted before
the statement of this corollary, $G$ is locally equivalent to a graph $H$ in
which case 4 of Proposition \ref{parallels} holds. Let $\beta
:M[IAS(G)]\rightarrow M[IAS(H)]$ be a compatible isomorphism induced by a local
equivalence between $G$ and $H$. We have just seen that there are
automorphisms $\beta_{a}$ and $\beta_{b}$ of $M[IAS(H)]$, which satisfy parts
2(a) and 2(b) of the statement for $H$ (respectively). It follows that the
compositions $\beta^{-1}\beta_{a}\beta$ and $\beta^{-1}\beta_{b}\beta$ satisfy
the statement for $G$.
\end{proof}

The familiar idea of parallel reduction in matroid theory is simply to delete
one of a pair of parallels. It makes little difference which of the two
parallels is deleted, because the identity map of the ground set defines an
isomorphism between the two resulting matroids. We would like to define an
analogous notion of \textquotedblleft parallel reduction\textquotedblright%
\ for isotropic matroids, but regrettably it cannot be quite so simple. There
are two complications here that do not affect ordinary matroidal parallel reduction:

1. To obtain an isotropic minor of an isotropic matroid we cannot simply
delete an element. We must remove a whole vertex triple, by deleting two
elements and contracting the third.

2. Corollary \ref{parallelcor} tells us that choosing which parallel to
delete from $M[IAS(G)]$, and which element of that vertex triple to contract,
will not affect the resulting isotropic matroid up to isomorphism. However,
such an isomorphism need not be defined by the identity map of $W(G) $. For
instance, if $v$ is unlooped and $w\in N(v)$ then as noted in connection with
Theorem \ref{isominorspec}, when we contract $\chi_{G}(v)$ the resulting
isotropic minor is isomorphic to $M[IAS(((G_{s}^{v})_{s}^{w})_{s}^{v}-v)]$,
and an isomorphism involves changing the $\phi,\chi,\psi$ designations of some
matroid elements. On the other hand, if we contract $\phi_{G}(v)$ then the
resulting isotropic minor is identical to $M[IAS(G-v)]$.

Considering these complications, we always prefer to contract a $\phi$
element; consequently we always prefer to delete a parallel that is not a
$\phi$ element. (Note that it is impossible for two parallels to both be
$\phi$ elements, because no two columns of an identity matrix are the same.)
The following definition reflects these preferences.

\begin{definition}
\label{isoparallel}Let $G$ be a looped simple graph, and suppose $\rho$ and
$\sigma$ are distinct, parallel elements of $M[IAS(G)]$ such that $\rho$ is
not a $\phi_{G}$ element. An \emph{isotropic parallel reduction }of
$M[IAS(G)]$ corresponding to the pair $\rho,\sigma$ is an isotropic minor
obtained by contracting the $\phi_{G}$ element of the vertex triple that
contains $\rho$, and deleting both $\rho$ and the third element of that vertex triple.
\end{definition}

\begin{definition}
Let $G$ be a looped simple graph. A \emph{pendant-twin reduction} of $G$ is a
graph obtained from $G$ in one of the following ways:

\begin{enumerate}
\item Delete an isolated vertex.

\item Delete a twin vertex (adjacent or nonadjacent).

\item Delete a vertex of degree 1.
\end{enumerate}
\end{definition}

Proposition \ref{parallels} immediately implies the following.

\begin{corollary}
\label{isopar}The isotropic parallel reductions of $M[IAS(G)]$ are the
isotropic matroids of pendant-twin reductions of $G$.
\end{corollary}

Applying Corollary \ref{isopar} repeatedly, we deduce the following.

\begin{corollary}
\label{disther}Let $G$ be a looped simple graph, with its vertices listed in
order $v_{1},\ldots,v_{n}$. Then the following statements are equivalent:

\item
\begin{enumerate}
\item There is a sequence of $n-1$ pendant-twin reductions that begins with
$G$, in which the $i^{th}$ reduction involves removing the vertex $v_{i}$.

\item There is a sequence of $n-1$ isotropic parallel reductions that begins
with $M[IAS(G)]$, in which the $i^{th}$ reduction involves removing the vertex
$v_{i}$.
\end{enumerate}
\end{corollary}

If $G$ satisfies Corollary \ref{disther} then we refer to the two sequences of
reductions as \emph{resolutions}, the first a pendant-twin resolution of $G$
and the second an isotropic parallel resolution of $M[IAS(G)]$. A connected
graph that admits such resolutions is called \emph{distance hereditary}
\cite{BM}. Corollary \ref{disther} gives us a matroidal characterization of arbitrary distance hereditary graphs:
$M[IAS(G)]$ has an isotropic parallel resolution if and only if the connected components of $G$ are all distance hereditary.

Results connected with Corollary \ref{disther} have appeared in the literature before, in different contexts. Bouchet proved an equivalent version of Corollary \ref{disther} involving isotropic systems that are ``totally decomposable'' \cite[Corollary 3.3]{Btree}. A special case was discussed by Ellis-Monaghan and Sarmiento \cite{EMS}, who proved that if a distance hereditary graph has a pendant-twin resolution without any adjacent twin reduction, then it is the interlacement graph of a medial graph of a series-parallel graph.

We should emphasize that Corollary \ref{disther} does not assert that $M[IAS(G)]$ is a series-parallel matroid in the usual sense. Indeed, if $G$ has a connected component with three or more vertices then $M[IAS(G)]$ is not regular \cite{Tnewnew}, so it is certainly not series-parallel.

\section{Forests} \label{sec:forests}

Theorem \ref{forests} follows directly from two results that are already
known. One is the equivalence between parts 1 and 2 of Theorem \ref{theory},
and the other is a theorem of Bouchet \cite{Btree}, who verified a conjecture
of Mulder by proving that locally equivalent trees are isomorphic. Bouchet's
proof of Mulder's conjecture involves Cunningham's theory of split
decompositions \cite{Cu}; we provide an alternative argument that involves
isotropic parallel reductions instead.

The first step in this alternative argument is a special case of Proposition
\ref{parallels}.

\begin{proposition}
\label{parallelfor}Let $G$ be a forest. Two non-loop elements of $M[IAS(G)]$
are parallel if and only if they fall into one of these two categories:

\begin{enumerate}
\item The two non-loop elements of the vertex triple of an isolated vertex are parallel.

\item If $v\neq w\in V(G)$ and $N_{G}(v)=\{w\}$ then $\chi_{G}(v)$ and
$\phi_{G}(w)$ are parallel.
\end{enumerate}
\end{proposition}

\begin{proof}
This follows from Proposition \ref{parallels} because a forest has no looped
vertex, and no twins of degree $>1$.
\end{proof}

Suppose now that $G$ and $H$ are forests, and $M[IAS(G)] \cong M[IAS(H)]$. Then
$\left\vert V(G)\right\vert =$ $\left\vert V(H)\right\vert $ and as stated in Theorem~\ref{theory}, there is a compatible isomorphism $\beta:M[IAS(G)]\rightarrow
M[IAS(H)]$. As $\beta$ is a compatible isomorphism, there is an associated
bijection $\beta:V(G)\rightarrow V(H)$ such that for each $v\in V(G)$, $\beta$
maps the vertex triple $\tau_{G}(v)$ to $\tau_{H}(\beta(v))$.

Notice that Theorem \ref{forests} does not require any particular connection
between a matroid isomorphism $\beta:M[IAS(G)]\rightarrow M[IAS(H)]$ and a
graph isomorphism $G \cong H$. It is convenient to prove a slightly
stronger statement, which does require a connection: namely, if $\beta
:M[IAS(G)]\rightarrow M[IAS(H)]$ is a compatible isomorphism then there is a
bijection between $V(G)$ and $V(H)$ which defines a graph isomorphism 
$G \cong H$ and also agrees with the bijection $\beta:V(G)\rightarrow V(H)$ at
every vertex where $\beta(\phi_{G}(v))=\phi_{H}(\beta(v))$. During the
argument we refer to this statement as the \emph{strong form} of Theorem
\ref{forests}.

If $\left\vert V(G)\right\vert =\left\vert V(H)\right\vert =0$ the theorem is
satisfied vacuously. The argument proceeds using induction.

If $G$ has an isolated vertex $v$ then every entry of the $\chi_{G}(v)$ column of $IAS(G)$ is $0$, so $\chi_{G}(v)$ is a loop of $M[IAS(G)]$. Then
$\beta(\chi_{G}(v))$ is a loop of $M[IAS(H)]$, so necessarily $\beta(v)$ is
isolated in $H$ and $\beta(\chi_{G}(v))=\chi_{H}(\beta(v))$. As $\beta$
respects vertex triples, $\beta(\tau_{G}(v))=\tau_{H}(\beta(v))$. Of course
the submatroids $M[IAS(G)]-\tau_{G}(v)$ and $M[IAS(H)]-\tau_{H}(\beta(v))$ are
isomorphic, as we may simply restrict $\beta$. As $v$ and $\beta(v)$ are isolated, one need only look at the matrices $IAS(G)$ and $IAS(H)$ to see that $M[IAS(G)]-\tau_{G}(v)$ and $M[IAS(H)]-\tau_{H}(\beta(v))$ are $M[IAS(G-v)]$ and $M[IAS(H-\beta(v))]$, respectively. Consequently the inductive hypothesis tells us that $G-v\cong H-\beta(v)$. As
$v$ and $\beta(v)$ are both isolated, it follows that $G\cong H$. Moreover,
this graph isomorphism is given by a bijection that agrees with the
isomorphism $G-v\cong H-\beta(v)$ given by the induction hypothesis, and it
matches $v$ to $\beta(v)$, so it satisfies the strong form of the theorem.

If $G$ has no isolated vertex it has a vertex $v$ with precisely one neighbor, $w$. Then $\chi_{G}(v)$ is parallel to $\phi_{G}(w)$ in $M[IAS(G)]$, so $\beta(\chi_{G}(v))$ is parallel to $\beta(\phi_{G}(w))$ in $M[IAS(H)]$. As $\beta$ respects vertex triples, $\beta(\chi_{G}(v))$ and $\beta(\phi_{G}(w))$ cannot fall under case 1 of Proposition \ref{parallelfor}; they must fall
under case 2. Consequently, $\beta(\{\chi_{G}(v)$, $\phi_{G}(w)\})$ is either $\{\chi_{H}(\beta(v))$, $\phi_{H}(\beta(w))\}$ or $\{\phi_{H}(\beta(v)),\chi_{H}(\beta(w))\}$.

Composing $\beta$ with the automorphism of $M[IAS(G)]$ from case 2(a)
of Corollary \ref{parallelcor} if necessary, we may presume that $\beta
(\chi_{G}(v))$ $=\chi_{H}(\beta(v))$ and $\beta(\phi_{G}(w))=$ $\phi_{H}(\beta(w))$.
Then $\chi_{H}(\beta(v))$ and $\phi_{H}(\beta(w))$ are parallel in
$M[IAS(H)]$, so it must be that $\beta(v)$ is pendant on $\beta(w)$ in $H$.
Composing with the automorphism of $M[IAS(G)]$ mentioned in case 2(b) of
Corollary \ref{parallelcor} if necessary, we may also presume that $\beta
(\phi_{G}(v))=\phi_{H}(\beta(v))$.

Then $\beta$ induces an isomorphism between the isotropic minors
\begin{gather*}
M[IAS(G-v)]=(M[IAS(G)]/\phi_{G}(v))-\chi_{G}(v)-\psi_{G}(v)\text{ and }\\
M[IAS(H-\beta(v))]=(M[IAS(H)]/\phi_{H}(\beta(v)))-\chi_{H}(\beta(v))-\psi
_{H}(\beta(v))\text{.}%
\end{gather*}
The inductive hypothesis tells us that there is a bijection between $V(G-v)$
and $V(H-\beta(v))$, which defines a graph isomorphism and agrees with the
bijection defined by $\beta$ at every vertex $x\in V(G-v)$ where $\beta
(\phi_{G}(x))=\phi_{H}(\beta(x))$. In particular, the isomorphism matches $w$
to $\beta(w)$. As $v$ and $\beta(v)$ are pendant on $w$ and $\beta(w)$
respectively, it follows that we can extend that isomorphism to an isomorphism
$G\cong H$, which matches $v$ to $\beta(v)$. Clearly this isomorphism also
satisfies the strong form of the theorem.

\subsection*{Acknowledgements}
We thank the referee for careful readings of preliminary versions of the paper, and many useful comments.

\end{document}